\documentclass[12pt,a4paper]{amsart}
\usepackage{mathrsfs, amssymb,amsthm, amsfonts, amsmath}
\usepackage{fullpage}
\usepackage{euscript}
\usepackage{tabularx,multirow}
\usepackage[all]{xy}
\usepackage{upgreek} 
\usepackage[shortlabels]{enumitem}
\usepackage{setspace}
\usepackage{tabularx,multirow}

\usepackage{hyperref}

 \usepackage[dvipsnames]{xcolor}
\usepackage{graphicx}

\newcommand{\Lie}{ \operatorname{{Lie}}}

\newcommand{\id}{ \operatorname{{id}}}

\newcommand{\sing}{\operatorname{{sing}}}
\newcommand{\Ad}{\operatorname{{Ad}}}
\newcommand{\Aut}{\operatorname{{Aut}}}
\newcommand{\Stab}{\operatorname{{Stab}}}
\newcommand{\GL}{\operatorname{GL}}
\newcommand{\SL}{\operatorname{SL}}

\newcommand{\PGL}{\operatorname{PGL}}

\newcommand{\Pic}{\operatorname{{Pic}}}

\newcommand{\W}{\operatorname{W}}

\newcommand{\Orb}{\operatorname{O}}

\newcommand{\aaa}{\mathrm{a}}
\newcommand{\s}{\mathrm{s}}
\newcommand{\m}{\mathrm{m}}

\newcommand{\GG}{{\mathbb G}}
\newcommand{\Ga}{{\mathbb G}_{\mathrm{a}}}
\newcommand{\Gm}{{\mathbb G}_{\m}}
\newcommand{\ZZ}{{\mathbb{Z}}}
\newcommand{\TT}{{\mathbb{T}}}

\newcommand{\PP}{{\mathbb{P}}}
\newcommand{\FF}{{\mathbb{F}}}
\newcommand{\A}{{\mathbb{A}}}

\newcommand{\CC}{{\mathbb{C}}}

\newcommand{\Sym}{{\mathfrak{S}}}

\newcommand{\fD}{{\mathfrak{D}}}

\newcommand{\fh}{{\mathfrak{h}}}
\newcommand{\fg}{{\mathfrak{g}}}
\newcommand{\ff}{{\mathfrak{f}}}

\newcommand{\cZ}{{{\mathcal{Z}}}}

\newcommand{\cO}{{\mathcal{O}}}

\newcommand{\cC}{\mathcal{C}}

\newcommand{\LLL}{{\mathscr{L}}}

\newcommand{\SSS}{{\mathscr{S}}}

\renewcommand{\emptyset}{\varnothing}

\theoremstyle{plain}

\newtheorem{thm}{Theorem}[section]
\newtheorem{cor}[thm]{Corollary}
\newtheorem{lem}[thm]{Lemma}
\newtheorem{prop}[thm]{Proposition}
\theoremstyle{definition}

\newtheorem{nota}[thm]{Notation}

\newtheorem{rems}[thm]{Remarks}
\newtheorem{exa}[thm]{Example}

\newcounter{caseinproof}
0

\makeatletter
\@namedef{subjclassname@2020}{\textup{2020} Mathematics Subject Classification}
\makeatother
\title[Involutions of Fano--Mukai fourfolds]{Central Weyl involutions on Fano--Mukai fourfolds of genus $10$}

\author{Mikhail Zaidenberg}
\address{Univ. Grenoble Alpes, CNRS, IF, 38000 Grenoble, France}
\email{Mikhail.Zaidenberg@univ-grenoble-alpes.fr}

\dedicatory{In memory of Jean-Pierre Demailly}

\begin{document}

\begin{abstract} It is known that
every Fano--Mukai fourfold $X$ of genus 10 is acted upon by an involution $\tau$ 
which comes from the center of the Weyl group 
of the simple algebraic group of type ${\rm G}_2$, see \cite{PZ18, PZ22}. 
This involution is uniquely defined up to conjugation in the group $\Aut(X)$. 
In this note we describe the set of fixed points 
of $\tau$ and the surface scroll swept out by the $\tau$-invariant lines.
\end{abstract}

\date{}

\maketitle
 
\setcounter{tocdepth}{2}\tableofcontents

\section*{Introduction} 

Let $V$ be a Fano--Mukai fourfold of genus 10 over $\CC$, that is, a smooth 
Fano fourfold of index 2 and of degree 18. A classification of the automorphism groups 
of these fourfolds was started in \cite{PZ18} and was completed
in~\cite[Theorem~A]{PZ22}. 
In particular, every $V$ is acted upon by an involution $\tau\in \Aut(V)\setminus\Aut^0(V)$ 
defined uniquely up to conjugation in $\Aut(V)$. It interchanges any pair of disjoint 
$\Aut^0(V)$-invariant cubic surface cones on $V$, see ~\cite[Proposition~2.11]{PZ22} 
and Corollary~\ref{cor:aut-GL2} below. 

This note is devoted to the following problem, see~\cite[Problem~15.4]{PZ18}.

\smallskip

\noindent \emph{Given a Fano--Mukai fourfold $V$ of
genus $10$, describe the involutions acting on $V$ and interchanging the pairs 
of $\Aut^0(V)$-invariant cubic cones. }

\smallskip

  Let $\fg$ be the Lie algebra of the simple algebraic group $G$ of type ${\rm G}_2$. 
  The adjoint variety $\Omega$ of $G$ is the unique closed orbit in the 
  projectivized adjoint representation of $G$ on $\PP\fg\cong\PP^{13}$. 
  We have $\dim\Omega=5$.
Every Fano--Mukai fourfold $V$ of
genus 10 admits a realization as a hyperplane section 
$\Omega\cap\PP^{12}$, see~\cite[Theorem~2]{Muk89}. 
Such a hyperplane section is unique up to the $G$-action on $\Omega$. 
Under this realization, $\Aut(V)$ coincides with the stabilizer of $V$ in $G$ and $\tau$ 
extends to an element  of order 2 of $G$. 

Let $T$ be a maximal torus of $G$  and $N_G(T)$ be the normalizer of $T$ in $G$.  
The Weyl group $\W=N_G(T)/T$ is isomorphic to the dihedral group $\fD_6$; 
there is a splitting $N_G(T)\cong T\rtimes \W$~\cite[Theorem~A]{AH17}. 
Up to a choice of such a splitting and up to conjugation in $G$, 
$\tau$ can be identified as the unique element of order 2 from the center of $W$.  
This is why we call $\tau$ a  \textit{central Weyl involution}.

The main result of the present note is the following theorem.

\begin{thm}\label{thm:mthm}
Let $V$ be a Fano--Mukai fourfold  of genus $10$ half-anticanonically embedded in $\PP^{12}$, 
and let $\tau\in \Aut(V)\setminus\Aut^0(V)$ be an involution. 
Then $\tau\in G$ is a central Weyl involution. 
The fixed point set $V^\tau$ is a union of two disjoint smooth rational sextic curves 
$E^+$ and $E^-$ such that 
\[\langle E^+\rangle=\PP^5,\quad  N_{E^+/\PP^{5}} = \cO_{\PP^1}(8)^{\oplus 2} 
\oplus \cO_{\PP^1}(9)^{\oplus 2}\quad\text{and}\quad \langle E^-\rangle=\PP^6,
\quad N_{E^-/\PP^{6}} = \cO_{\PP^1}(8)^{\oplus 5}.\]
Furthermore, 
there is a surface scroll $\Pi=\Pi(V,\tau)$ in $V$ verifying the following.
\begin{enumerate}\item[(i)] Each ruling of $\Pi$ is $\tau$-invariant and each 
$\tau$-invariant line on $V$ is a ruling of $\Pi$.
\item[(ii)] $\Pi$ has degree $12$, is linearly nondegenerate in $\PP^{12}$ 
and is isomorphic to $\PP^1\times\PP^1$ embedded in $\PP^{12}$ by a linear system of type $(1,6)$. 
\item[(iii)] An isomorphism $\Pi\cong\PP^1\times\PP^1$ sends the curves 
$E^\pm$ into constant sections of the first projection ${\rm pr}_1\colon\PP^1\times\PP^1\to\PP^1$. 
\end{enumerate}
\end{thm}

Due to (i) we call $\Pi$ the  \emph{scroll in $\tau$-invariant lines}.
Notice that neither $\tau$ nor $\Pi$ are $\Aut^0(V)$-invariant. 

Results of \cite{PZ22} were extended in \cite{BM22} to the automorphism groups of smooth hyperplane sections of other generalized flag varieties $G/P$; most of these hyperplane sections also possess Weyl involutions. It is worth to obtain a geometric description similar to that of Theorem \ref{thm:mthm} in this more general setting; see e.g. \cite[Proposition~24]{DM22} for some results in this direction.

We end this introduction with the following open question. Recall that the Fano--Mukai fourfolds of genus 10 are rational; moreover, any of these is a completion of $\CC^4$, 
see \cite[Theorem 1.1]{PZ18}. 

\smallskip

\noindent {\bf Question.} \emph{Is any central Weyl involution acting on a  Fano--Mukai fourfold $V$ of genus $10$ linearizable, that is, conjugate to a linear involution of $\PP^4$ via a birational map $V\dasharrow\PP^4$?}

\medskip

 {\bf Acknowledgments.} This paper is based on joint articles \cite{PZ18, PZ22, PZ23} of Yuri Prokhorov 
 and the author.  We are grateful to Yuri Prokhorov for useful discussions; the ideas of several proofs are due to him. 
 Our thanks also due to Ciro Ciliberto for his kind assistance, and especially for producing several 
 amazing examples of singular surface scrolls. 
 We are grateful to the anonymous referee, whose comments were relevant and particularly helpful in avoiding inaccuracies, clarifying arguments and improving style.

\medskip

\section{Preliminaries}
\subsection{Lines and cubic scrolls on Fano--Mukai fourfolds} 
We gather here various results from \cite{KR13}, \cite{PZ18}, \cite{PZ22} and \cite{PZ23} 
on the geometry of Fano--Mukai fourfolds 
that will be used in the sequel. 

Fix a Fano--Mukai fourfold $V$ of
genus 10 together with an embedding $V\hookrightarrow\PP^{12}$ 
by the half-anticanonical system.
Let $\Sigma=\Sigma(V)$ be the Hilbert scheme of lines on $V$ and 
$\SSS=\SSS(V)$ be the Hilbert scheme of cubic surface scrolls  on $V$. 
By a \emph{cubic cone} we mean the projective cone in $\PP^4$ over 
a rational twisted cubic curve in  $\PP^3$. 
Given an involution $\tau\in\Aut(V)\setminus\Aut^0(V)$ 
we denote by the same letter the induced involutions acting 
on $\Sigma$ and on $\SSS$. The Hilbert scheme of lines $\Sigma(V)$
can be described as follows. 

\begin{prop}\label{prop:prelim-lines} $\,$
\begin{enumerate}
\item[{\rm (a)}]
$\Sigma(V)$ is a smooth hyperplane section of $\PP^2\times\PP^2$ embedded in $\PP^7$ via the Segre embedding. Up to an automorphism of $\PP^2\times\PP^2$ one can choose this section in such a way that the action of $\tau$ on $\Sigma$ is induced by the involution on  $\PP^2\times\PP^2$ interchanging the factors. 
\item[{\rm (b)}]
There is 
the diagram
\begin{equation}
\label{equation-universal-family-V}
\vcenter{
\xymatrix@R=10pt{
&\LLL(V)\ar[dl]_{\rho}\ar[dr]^{s}&
\\
\Sigma(V)&&V
}}
\end{equation}
where $\LLL(V)\subset\Sigma(V)\times V$ is the incidence relation between lines and points on $V$. The  $\PP^1$-bundle $\rho\colon \LLL(V)\to\Sigma(V)$ is the universal family of lines on $V$. The action of $\tau$ on $\Sigma(V)\times V$ leaves $\LLL(V)$ invariant and respects the $\PP^1$-bundle structure. 
\item[{\rm (c)}]
The map $s\colon\LLL(V)\to V$ in~\eqref{equation-universal-family-V} is a generically
finite morphism of degree $3$. So, $V$ is covered by lines and through a general point of $V$ pass precisely $3$ lines.  
\end{enumerate}
\end{prop}
\begin{proof} See 
\begin{itemize}
\item \cite[Proposition~2]{KR13}, \cite[Theorem~9.1(a), Lemma~9.5.1 and its proof]{PZ18} for (a), 
\item \cite[Proposition~8.2(d), (8.2.2)]{PZ18} for (b), and
\item \cite[Proposition~8.2(d)]{PZ18} for (c).
\end{itemize}
\end{proof}
Next we describe the Hilbert scheme $\SSS(V)$ of cubic surface scrolls on $V$.
\begin{prop}\label{prop:prelim-scrolls} $\,$
\begin{enumerate}
\item[{\rm (d)}]
The Hilbert scheme $\SSS=\SSS(V)$ of cubic scrolls on $V$ has exactly  two irreducible components $\SSS_i\cong\PP^2$, $i=1,2$. These components are disjoint and interchanged by $\tau$.
\item[{\rm (e)}]\label{item-cohomology} $H^4(V,\ZZ)=\ZZ [S_1]\oplus\ZZ [S_2]$, 
where the $S_i\in\SSS_i$ satisfy the relations $[S_1]^2=[S_2]^2=1$ and $[S_1]\cdot [S_2]=0$. 
\item[{\rm (f)}]
Two scrolls $S_1\in\SSS_1$ and $S_2\in\SSS_2$ from different components of the Hilbert scheme $\SSS$ either are disjoint or contain a unique common ruling. Any line $l$ on $V$ is the unique common ruling of exactly two cubic scrolls $S_i(l) \in \SSS_i$, $i=1,2$.
The morphisms
\[{\rm pr}_i\colon\Sigma\to \SSS_i=\PP^2,\quad l\mapsto S_i(l),\,\,\, i=1,2\]
coincide with the standard projections in (i):
\[{\rm pr}_i\colon\Sigma\subset\PP^2\times\PP^2\to\PP^2,\,\,\, i=1,2.\] 
The fiber of ${\rm pr}_i$ over $S \in \SSS_i$ is the line $\Lambda(S)$ on $\Sigma$
which is the Hilbert scheme of the rulings of $S$.
\item[{\rm (g)}]
For a cubic scroll $S$ on $V$ 
the Hilbert scheme
$\Sigma(S)\subset\Sigma$ of lines on $V$ meeting $S$ is the pull-back of a line in 
$\PP^2=\SSS_j$ under the second projection ${\rm pr}_j$, $j\neq i$. 
One has $\Sigma(S) \cong \mathbb{F}_1$ and $\Lambda(S)$ is the exceptional section of $\Sigma(S)$.
\item[{\rm (h)}] Any cubic scroll $S$ on $V$ coincides with the singular locus of a unique hyperplane section $A_S$ of $V$, where $A_S$ is the union of lines on $V$ meeting $S$. Any line contained in $A_S$ meets $S$. Through any point of $A_S\setminus S$ passes a unique line on $V$ which meets $S$.
\end{enumerate}
\end{prop}
\begin{proof} See 
\begin{itemize}
\item \cite[Proposition~1]{KR13}, \cite[Theorem~9.1(b)]{PZ18} for (d), 
\item \cite[Proposition~9.6]{PZ18} for (e),
\item \cite[Corollaries~9.7.3, 9.7.4, and 9.10.1]{PZ18} for (f),
\item \cite[Theorem~9.1(a), Lemma 9.7.2]{PZ18} for (g), and
\item \cite[Lemma~9.2]{PZ18} for (h).
\end{itemize}
\end{proof} 
Finally, for cubic surface cones on $V$ we have the following. 
\begin{prop}\label{prop:prelim-cones} $\,$
\begin{enumerate}
\item[{\rm (i)}] Any cubic cone $S$ on $V$ is contained in the branching divisor $\mathcal{B}\subset V$ of $s$. 
\item[{\rm (j)}]
 Any line on $V$ passing through the vertex of a cubic cone $S$ lies on $S$ (we call it a ruling of $S$). If a point $v\in\mathcal{B}$ is different from the vertices of cubic cones on $V$ then the number of lines  passing through $v$ equals either 1 or 2. Through any point on $V\setminus \mathcal{B}$ passes exactly 3 distinct lines on $V$. 
\item[{\rm (k)}]
The Hilbert scheme of cubic cones on $V$ is either finite or one-dimensional. 
The number of $\Aut^0(V)$-invariant cubic cones on $V$ is finite and $V$ contains at least one pair $(S_1, S_2)$ of disjoint $\Aut^0(V)$-invariant cubic cones. 
\item[{\rm (l)}] $\bigcap_{S\in\SSS_i} A_S$ is the union of vertices of cubic cones in $\SSS_i$ and $\bigcap_{S\in\SSS_1\cup\SSS_2} A_S=\emptyset$, see Proposition \ref{prop:prelim-scrolls} $\rm (h)$.
\end{enumerate}
\end{prop}
\begin{proof} See 
\begin{itemize}
\item \cite[Lemma~9.4]{PZ18} for (i),
\item \cite[Proposition~8.2(e)]{PZ18} for (j),
\item \cite[Lemma~1.8, Proposition~2.11]{PZ23} for (k), and
\item \cite[Lemma~1.9]{PZ23} for (l).
\end{itemize}
\end{proof} 

\subsection{Mukai ${\mathrm G}_2$-presentation}\label{ss:Mukai} 
Let $\Omega\subset\PP^{13}$ be the adjoint variety  of  the simple algebraic group
$G$ of type ${\rm G}_2$, that is, the minimal nilpotent orbit $\Omega=G/P$ of the adjoint action of $G$ on $\PP\fg=\PP^{13}$, where $\fg={\Lie}(G)$ and $P\subset G$ is a parabolic subgroup of dimension 9. Recall that $\Omega$ is a Fano--Mukai fivefold embedded in $\PP^{13}$ by the linear system $|-\frac{1}{3}K_\Omega|$.
Using the duality given by the Killing form on $\fg$ we identify $\PP\fg$ with its dual projective space $\PP\fg^\vee$. Abusing notation we let $h^\bot$ denote the hyperplane in $\PP\fg$ orthogonal to a vector $h\in\fg\setminus\{0\}$. Letting $\PP h$ stand for the vector line $\CC h\in\PP\fg$ we let $V(h)=h^\bot\cap\Omega$ denote the corresponding hyperplane section of $\Omega$. According to~\cite[Theorem~2]{Muk89} any Fano--Mukai fourfold $V$ of genus 10 admits a \textit{Mukai realization} as a smooth hyperplane section $V=V(h)$ of $\Omega$. 

The dual variety $D_{\ell}=\Omega^*$ is a sextic hypersurface in $\PP^{13}$ whose points $\PP h\in D_\ell$ correspond to the singular hyperplane sections $V(h)$ of $\Omega$. Thus, the smooth hyperplane sections $V(h)$ are parameterized by the affine variety $\PP^{13}\setminus D_{\ell}$. The group $G$ acts on $\PP^{13}$ via the projective adjoint representation. Each $G$-orbit in $\PP^{13}\setminus D_{\ell}$ corresponds to a class of isomorphic smooth hyperplane sections of $\Omega$, that is, to a class of isomorphic Fano--Mukai fourfolds of genus 10. Indeed, according to Mukai's theorem~\cite[Theorem~0.9]{Muk89}, any isomorphism between two smooth hyperplane sections of $\Omega$ extends to an automorphism of $\Omega$ and the group $\Aut(\Omega)$ coincides with $G|_\Omega$. For a smooth hyperplane section $V=V(h)$ of $\Omega$ one has (see~\cite[Corollary~2.3.2]{PZ22})
\begin{equation}\label{eq:aut=stab}\Aut(V)=\Stab_G(\PP h).\end{equation}

Given a maximal torus $T$ of $G$ we let $\fh={\Lie}(T)$ be the corresponding Cartan subalgebra in $\fg$ and $\Delta\subset\fh^*$ be the corresponding root system. For a root $\alpha\in\Delta$ we let $\fg_\alpha$ be  the corresponding one-dimensional root subspace of $\fg$. The points $\PP\fg_{\alpha_i}\in\PP\fg$, $i=1,\ldots,6$ that correspond to  the 6 long roots $\alpha_1,\ldots,\alpha_6$ lie on $\Omega$, while the points  $\PP\fg_{\alpha_i}$, $i=7,\ldots,12$ that correspond to  the 6 short roots $\alpha_7,\ldots,\alpha_{12}$  lie on another nilpotent orbit $\Omega_\s$ of $G$. The dual hypersurface $D_\s=({\overline{\Omega_\s}})^*$ of $\overline{\Omega_\s}$ is also a  sextic. Thus, $D_\ell$  is dual to the $G$-orbit $\Omega$ of the projectivized long roots, while  $D_\s$ is dual to the orbit closure $\overline{\Omega_\s}$ where $\Omega_\s$ is the $G$-orbit of the projectivized short roots. This justifies our notation. 

Let $(D_t)_{t\in\PP^1}$ be the pencil of sextic hypersurfaces  in $\PP^{13}=\PP\fg$ generated by $D_\ell$ and $D_\s$. Abusing notation, we let $\ell,\s\in\PP^1$ stand for the values of $t$ which correspond to $D_\ell$ and $D_\s$, respectively. 
Any sextic hypersurface $D_t$ is invariant under the adjoint $G$-action on $\PP\fg$. Thus, any $G$-orbit in $\PP\fg$ is contained in $D_t$ for a suitable $t\in\PP^1$. 
There is the following description of  $G$-orbits in $\PP\fg$.

\begin{thm}[{\rm \cite[Lemma~1]{KR13},~\cite[Proposition~1.4]{PZ22}}]\label{thm:orbits}$\,$
\begin{enumerate}
\item[(i)] 
The base locus $D_\ell\cap D_\s$ of the pencil $(D_t)_{t\in\PP^1}$ coincides with the projectivized nilpotent cone of $\fg$. Thus, all the nilpotent $G$-orbits in $\PP\fg$, including $\Omega$ and $\Omega_\s$, are contained in $D_\ell\cap D_\s$.
\item[(ii)]
For any $t\in\PP^1\setminus\{\ell,\s\}$ the complement $D_t\setminus D_\ell$ coincides with the orbit of a point $\PP s\in\PP\fg$
where $s\in \fg$ is a regular semisimple element. 
\item[(iii)] $D_\s\setminus D_\ell$ is the union of two $G$-orbits $\Omega^\aaa$ and $\Omega^{\mathrm{sr}}$, where $\Omega^\aaa$ of dimension $12$ is open and dense in $D_\s\setminus D_\ell$ and 
$\Omega^{\mathrm{sr}}$ of dimension $10$ is closed in $D_\s\setminus D_\ell$. 
There exists a pair of commuting  elements $f,n\in\fg$, where $f$ is subregular semisimple, $n$ is nilpotent and $f+n$ is regular such that $\Omega^\aaa$ is the $G$-orbit of $\PP (f+n)$  and $\Omega^{\mathrm{sr}}$ is  the $G$-orbit of $\PP f$. \footnote{The upper index $\mathrm{sr}$  in $\Omega^{\mathrm{sr}}$ stands for ``subregular''.}
\end{enumerate}
\end{thm}

\subsection{Automorphism groups and involutions of Fano--Mukai fourfolds}

According to \eqref{eq:aut=stab} the classification of the automorphism groups of Fano--Mukai fourfolds of genus 10  is ultimately related to the classification of non-nilpotent $G$-orbits in $\PP\fg$. Indeed,  due to the existence of the Mukai presentation, the latter groups are stabilizers of the $G$-orbits in $\PP\fg\setminus D_\ell$. These orbits are $D_t\setminus D_\ell$, 
$t\in\PP^1\setminus\{\ell,\s\}$,
$\Omega^\aaa$ and $\Omega^{\mathrm{sr}}$, see Theorem \ref{thm:orbits}. Accordingly, we have  the following description. 

\begin{thm}[{\rm \cite[Theorem~1.3]{PZ18}, \cite[Proposition~1.5]{PZ22}}]\label{thm:Aut} Let $V=V(h)$ be a Fano--Mukai fourfold of genus $10$ under a Mukai realization, where $\PP h\in\PP\fg \setminus D_\ell$. Then the identity component $\Aut^0(V)$ is as follows:
\[\label{eq:aut0} \Aut^0(V(h))={\rm Stab}_G(\PP h)\cong\begin{cases} 
\GL(2,\CC),&\,\,\,
\,\,\, \PP h\in \Omega^{\mathrm sr}\\
\Ga\times\Gm,&\,\,\,
\,\,\, \PP h\in \Omega^\aaa\\
\Gm^2,&\,\,\,
\,\,\, \PP h\in \PP\fg\setminus (D_\ell\cup D_\s)=\bigcup_{t\neq\s,\ell} (D_t\setminus D_\ell)
\end{cases}\]
Consequently, there is a unique up to isomorphism Fano--Mukai fourfold $V$ of genus $10$ with 
$\Aut^0(V)\cong\GL(2,\CC)$ (resp. $\Aut^0(V)\cong\Ga\times\Gm$), 
while the isomorphism classes of Fano--Mukai fourfolds with $\Aut^0(V)\cong\Gm^2$ form a one-parameter family parameterized by the points $t\in\PP^1\setminus\{\ell,\s\}$.
\end{thm}

The geometry of a Fano--Mukai fourfold $V$ and the group $\Aut(V)$ are ultimately related. 

\begin{prop}[{\rm \cite[Proposition 11.1]{PZ18}}]\label{prop:cubic-cones} A Fano--Mukai fourfold $V$ carries 
\begin{itemize}\item
exactly $6$ cubic cones if $\Aut^0(V)\cong\Gm^2$, 
\item exactly $4$ cubic cones if $\Aut^0(V)\cong\Ga\times\Gm$,
\item a one-parameter family of cubic cones if $\Aut^0(V)\cong\GL(2,\CC)$. 
\end{itemize}
In the first two cases every cubic cone on $V$ is $\Aut^0(V)$-invariant. 
In the last case there are exactly two $\Aut^0(V)$-invariant cubic cones on $V$, they are disjoint and belong to different connected components $\SSS_1$ and $\SSS_2$ of $\SSS$. 
\end{prop}

Concerning the discrete part of the automorphism group $\Aut(V)$ and its action on $V$ we have the following results. 

\begin{thm}[{\rm \cite[Theorem~A, Proposition~3.11]{PZ22}}]\label{thm:discrete} $\,$
\begin{enumerate}
\item[(i)] 
If $\Aut^0(V)$ is one of the groups $\GL(2,\CC)$ and $\Ga\times\Gm$ then $\Aut(V)=\Aut^0(V)\rtimes\ZZ/2\ZZ$. The same holds in the case $\Aut^0(V)\cong \Gm^2$ except for $V_{\rm r}=V(h_{\rm r})$ with $\PP h_{\rm r}\in D_{\rm r}=3Q$ where $Q\subset\PP\fg$ is the quadric defined by the Killing  form. In the latter case $\Aut(V_{\rm r})\cong \Gm^2\rtimes\ZZ/6\ZZ$.  
\item[(ii)]
Let  $\tau$ be the generator of the factor $\ZZ/2\ZZ$ in the former cases and the unique order 2 element $\tau=\zeta^3$ in the latter case, where $\zeta$ is a generator of the factor $\ZZ/6\ZZ$. Then $\tau$ is an involution on $V$ which interchanges any pair of disjoint $\Aut^0(V)$-invariant cubic cones on $V$. There is precisely one such pair of cubic cones in the cases where $\Aut^0(V)$ is one of the groups $\GL(2,\CC)$ and $\Ga\times\Gm$, and precisely $3$  such pairs in the case $\Aut^0(V)\cong\Gm^2$. 
\item[(iii)] 
If $\Aut^0(V)\cong\Gm^2$ then the $6$ cubic cones  on $V$ are naturally arranged in a cycle $\sigma=(S_1,\ldots,S_6)$ where $S_i$ and $S_j$ are neighbors if and only if they share a common ruling, and are opposite (that is, $j=i+ 3 \mod 6$) if and only if they are disjoint, and therefore interchanged by the involution $\tau$. For $V=V_{\mathrm r}$
any element of the factor $\ZZ/6\ZZ$ of $\Aut(V_{\mathrm r})$ acts on $\sigma$ via a cyclic shift. 
\item[(iv)] The involution $\tau\in\Aut(V)\setminus\Aut^0(V)$ can be chosen in such a way that
the action of $\tau$ on $\Aut^0(V)$ by conjugation  is the inversion $g\mapsto g^{-1}$ if $\Aut^0(V)$ is one of the groups $\Gm^2$ and $\Ga\times\Gm$ and the Cartan involution $c\colon g\mapsto (g^t)^{-1}$ if $\Aut^0(V)=\GL(2,\CC)$, where $g\mapsto g^t$ stands for the matrix transposition. 
\end{enumerate}
\end{thm}

We need in the sequel the following elementary facts. 

\begin{prop}\label{prop:aut-GL2} $\,$
\begin{enumerate}\item[(i)] 
Any involution $\tau\in\Aut(\GL(2,\CC))\setminus {\rm Inn}(\GL(2,\CC))$ is conjugate in $\Aut(\GL(2,\CC))$ to the Cartan involution $c$.
\item[(ii)] Let $H$ be a group  admitting square roots, that is, for any $h\in H$ there exists $g\in H$ such that $h=g^2$. Consider the semidirect product $\tilde H=H\rtimes\ZZ/2\ZZ$ where the generator $c$ of the factor $\ZZ/2\ZZ$ acts on $H$ via inversion. Then for any $h\in H$ the element $ch\in \tilde H\setminus H$ has order 2 and is conjugate to $c$.
\item[(iii)] Let $\tilde H=H\rtimes\ZZ/2\ZZ$ where $H=\Ga\times\Gm$ and the generator $c$ of the factor $\ZZ/2\ZZ$ acts on $H$ via inversion. Then any element $\tau\in \tilde H\setminus H$ of order 2 is conjugate to $c$. 
\item[(iv)] Let $\tilde T=T\rtimes\ZZ/2\ZZ$ where $T=\Gm^2$ and the generator $c$ of the factor $\ZZ/2\ZZ$ acts on  $T$ via inversion. Then any element $\tau\in \tilde T\setminus T$ of order 2 is conjugate to $c$. 
\item[(v)] 
There exist exactly two different (conjugated) subgroups $H^\pm$ of $\GL(2,\CC)$ such that 
\begin{enumerate}
\item $H^\pm\cong\Ga\times\Gm$;
\item $H^\pm$ is invariant under the Cartan involution $c$; 
\item $c|_{H^\pm}$ is the inversion $g\mapsto g^{-1}$. 
\end{enumerate}
\end{enumerate}
\end{prop}

\begin{proof} (i) Let us first show the assertion of (i) for the derived subgroup $\SL(2,\CC)$ of $\GL(2,\CC)$. This group is invariant under the $\Aut(\GL(2,\CC))$-action and any automorphism of $\SL(2,\CC)$ extends to an automorphism of $\GL(2,\CC)$. 

It is known that $\Aut(\SL(2,\CC))\cong {\rm Inn}(\SL(2,\CC))\rtimes\ZZ/2\ZZ$, see~\cite[Theorem~10.6.10]{Pro06}. Since $c|_{\SL(2,\CC)}\in\Aut(\SL(2,\CC))\setminus {\rm Inn}(\SL(2,\CC))$, we may suppose that the factor  $\ZZ/2\ZZ$ is generated by $c$. Then $\tau=\Ad(a)\cdot c$ for some $a\in \SL(2,\CC)$ where $\Ad(a)\colon f\mapsto afa^{-1}$ for $f\in\SL(2,\CC)$. Assume that 
$\tau$ is a conjugate of $c$ in $\Aut(\SL(2,\CC))$, that is,  for some $b\in\SL(2,\CC)$ we have
\[\tau=\Ad(b)\cdot c\cdot\Ad(b^{-1})=\Ad(bb^t)\cdot c.\] 

Now, from $\tau^2=1$ and $\tau=\Ad(a)\cdot c$ we deduce that $a=a^t$, that is, $a$ is a symmetric matrix. Then $a$ can be written as $a=bb^t$ for some $b\in\SL(2,\CC)$. It follows by the preceding formula that $\tau=\Ad(\Ad(b))(c)$, that is, $\tau$ is conjugate to $c$  in $\Aut(\SL(2,\CC))$, as stated. 

Any automorphism $\alpha\in\Aut(\GL(2,\CC))$ is a composition of an inner automorphism and the multiplication by a character $\det^k$ of $\GL(2,\CC)$ for some $k\in\ZZ$~\cite[Ch.~4, \S~1]{Die55}. However, if $\alpha$ has a finite order then $k=0$. Now the preceding argument applied mutatis mutandis  shows that any element $\alpha\in\Aut(\GL(2,\CC))\setminus {\rm Inn}(\GL(2,\CC))$ of order 2 is conjugate to the Cartan involution $c$ in $\Aut(\GL(2,\CC))$. 
 
(ii) For $h\in H$ the element $ch\in \tilde H\setminus H$  satisfies $ch=h^{-1}c$. Therefore, for $g,h\in H$ the equality $ch=gcg^{-1}$ holds if and only if $g^2=h^{-1}$. Since $H=\Ga\times\Gm$ admits square roots, it follows that for any $h\in H$ there exists $g\in H$ such that $ch=gcg^{-1}$ in $\tilde H$, that is, $ch$ is a conjugate of $c$  in $\tilde H$. 

Now (iii)--(iv) follow since both groups $H=\Ga\times\Gm$ and $T=\Gm^2$ admit square roots. 
 
 (v) Let $H$ be a subgroup of $\GL(2,\CC)$ isomorphic to $\Ga\times\Gm$. Then $H$ is contained in a Borel subgroup $B$ of $\GL(2,\CC)$. The unipotent $\Ga$-factor of $H$ coincides with the unipotent radical $U$ of $B$ and the $\Gm$-factor of $H$ coincides with the center $z(B)=z(\GL(2,\CC))$. Hence $B={\rm Norm}_{\GL(2,\CC)}(U)$ is the unique Borel subgroup containing $H$.
 
 The induced action of the outer automorphism $c$ of $\GL(2,\CC)$ by an involution on the flag variety $\GL(2,\CC)/B=\PP^1$ is effective and has exactly two fixed points, say $B^\pm$. It follows by the preceding observations that there are exactly two $c$-invariant subgroups $H^{\pm}\subset B^\pm$ with desired properties. Let us find these subgroups explicitly. 
 
 The $c$-action on the Lie algebra $\mathfrak{gl}(2,\CC)$ is given by $h\mapsto -h^t$. 
 Up to proportionality, there are exactly two nonzero elements $n$ of $\mathfrak{gl}(2,\CC)$ satisfying $n^2=0$ and $c(n)=-n^t=\lambda n$, namely,
 \[n^\pm=\begin{pmatrix} 1 & \pm i\\ \pm i & 1\end{pmatrix}\]
 where in fact $\lambda=-1$. 
 Therefore, for the unipotent radical of the subgroup $H^{\pm}$ we have
 \[R_u(H^\pm)=\{\exp(tn^\pm),\quad t\in\CC\}=\left\{\,
 \begin{pmatrix} 1+t & \pm it\\ \pm it & 1-t
 \end{pmatrix},\quad t\in\CC\right\}\]
 and the $\Gm$-factor of $H^\pm$ coincides with the center of $\GL(2,\CC)$.
\end{proof}

Due to the next corollary the choice of a concrete involution $\tau\in\Aut(V)\setminus\Aut^0(V)$ in Theorem~\ref{thm:mthm} is irrelevant. Let $\Pi(\tau)$ be the scroll swept out by the $\tau$-invariant lines on $V$, cf. Proposition \ref{lem:scroll} below. 

\begin{cor}\label{cor:aut-GL2} Let $V$ be  a Fano--Mukai fourfold $V$ of genus $10$. Then
any two involutions  $\tau_1,\tau_2\in\Aut(V)\setminus\Aut^0(V)$ are conjugate in the group $\Aut(V)$ via an element of $\Aut^0(V)$ which sends the fixed point set $V^{\tau_1}$ and the scroll  in invariant lines $\Pi(\tau_1)$ to $V^{\tau_2}$ and $\Pi(\tau_2)$, respectively. Every involution $\tau\in\Aut(V)\setminus\Aut^0(V)$ interchanges any pair of disjoint $\Aut^0(V)$-invariant cubic cones on $V$. 
\end{cor}

\begin{proof} According to Theorem~\ref{thm:discrete}(i) $\Aut(V)=\Aut^0(V)\rtimes\ZZ/2\ZZ$ except for $V=V(h)$ with $\PP h\in D_t=3Q$. In the latter case $\Aut(V)\cong\Gm^2\rtimes\ZZ/6\ZZ$ and the factor $\ZZ/6\ZZ$ contains a unique  element $c=\zeta^3$ of order 2.  Due to~Theorem \ref{thm:discrete} (iv) one can choose the generator $c$ of the factor $\ZZ/2\ZZ$ and the order 2 element $c\in\ZZ/6\ZZ$ in the exceptional case so that the conjugation with $c$ acts via the inversion on $\Aut^0(V)$ if $\Aut^0(V)$ is an abelian group and acts  on $\Aut^0(V)$ via the Cartan involution if $\Aut^0(V)\cong\GL(2,\CC)$. Furthermore, $c$ interchanges any pair of disjoint $\Aut^0(V)$-invariant cubic cones on $V$, see Theorem \ref{thm:discrete}(ii). However,
by Proposition~\ref{prop:aut-GL2}(i), (iii) and (iv) every involution $\tau\in\Aut(V)\setminus\Aut^0(V)$ is conjugate to $c$. Now the assertion follows.
\end{proof}

\section{Scroll $\Pi$ in $\tau$-invariant lines} 
\subsection{The first properties}\label{ss:first}
We fix as before a Fano--Mukai fourfold $V$ of genus 10 and an involution $\tau\in\Aut(V)\setminus\Aut^0(V)$. Many objects considered in this section, including the scroll $\Pi$ in $\tau$-invariant lines, depend on the choice of $\tau$. To simplify the notation we do not mention this dependence explicitly. 

Let  $\Sigma=\Sigma(V)$ stand as before for the Hilbert scheme of lines on $V\subset\PP^{12}$ regarded as a subscheme of the Grassmannian of lines $\GG(1,12)$. Then the fixed point subscheme $C= \Sigma^\tau$ parameterizes the $\tau$-invariant lines on $V$. 
Let $\LLL\subset\Sigma\times V$ be   the line-point incidence relation on $V$, and let $\rho\colon\LLL\to\Sigma$ be the natural projection, see~\eqref{equation-universal-family-V}; it comes from the projection of the universal $\PP^1$-bundle over $\GG(1,12)$ restricted to $\Sigma$.
Letting $\tilde\Pi=\rho^{-1}(C)\subset\LLL$, the projection $\rho|_{\tilde\Pi}\colon\tilde\Pi\to C$ defines a $\tau$-invariant ruling on $\tilde\Pi$ over $C$. The image $\Pi=\Pi(V,\tau):=s(\tilde\Pi)$ under the second projection $s\colon\LLL\to V$  is the union of all the $\tau$-invariant lines on $V$.

\begin{prop}\label{lem:scroll}  Let $\Sigma\subset\PP^2\times\PP^2$ be equipped with the polarization induced by the Segre embedding $\PP^2\times\PP^2\hookrightarrow\PP^8$. Then the following hold.
\begin{enumerate}\item[(i)] $C$ is a smooth rational quartic curve on $\Sigma$;
\item[(ii)] $\tilde\Pi$ is a smooth rational ruled surface;
\item[(iii)] $\Pi$ is a rational surface scroll of degree $12$ on $V$;
\item[(iv)] each $\tau$-invariant line on $V$ is a ruling of $\Pi$, and vice versa.
\end{enumerate}
\end{prop}

\begin{proof} (i) 
 The threefold $\Sigma$ admits a realization as the projectivization of the variety of square matrices of order 3 and of rank 1 with zero trace \cite[11.2.2]{PZ18}. Such a matrix $A$ can be written as the tensor product $a\otimes b$ of two nonzero $3$-vectors $a=(x_1,x_2,x_3)$ and $b=(y_1,y_2,y_3)$ with zero trace 
 \[{\rm tr}(A)=x_1y_1+x_2y_2+x_3y_3=0.\] 
 The involution $\tau$ interchanges each pair of disjoint $\Aut^0(V)$-invariant cubic cones on $V$. Hence, $\tau$ also interchanges the connected components $\SSS_1$ and $\SSS_2$ of the Hilbert scheme $\SSS$ of cubic scrolls on $V$, that is, the factors of $\PP^2\times\PP^2=\SSS_1\times\SSS_2$. So, it interchanges every pair of disjoint cubic cones, see Proposition~\ref{prop:prelim-scrolls}(f). 
 Up to a choice of coordinates, the induced action of $\tau$ on $\Sigma$ results in the twist $a\leftrightarrow b$. 
 The fixed point set $C=\Sigma^\tau$ is the intersection $\Sigma\cap\Delta$ where $\Delta\subset\PP^2\times\PP^2$ is the diagonal. 
 The image of $A=a\otimes b$ in $\Sigma$ is fixed by $\tau$ if and only if $a$ and $b$ are proportional. The condition ${\rm tr}(A)=0$ leads in the latter case to the relations
\[x_1^2+x_2^2+x_3^2=0\quad\text{and}\quad y_1^2+y_2^2+y_3^2=0.\] 
It follows that $C=\Sigma^\tau$ is a smooth rational curve on $\PP^2\times\PP^2=\SSS_1\times\SSS_2$ of bidegree $(2,2)$ which projects isomorphically onto a smooth conic $C_i\subset\SSS_i$. Thus, $C$ is a conic in $\Delta\cong\PP^2$ and a quartic curve in $\PP^2\times\PP^2$.

(ii) The smooth morphism $\rho\colon\LLL\to\Sigma$ is the projection of a $\PP^1$-bundle. Hence (ii) follows. 

(iii) By virtue of (ii), $\Pi$ is an irreducible surface scroll on $V$ whose rulings are parameterized by $C$. 
The scroll $\Pi$ being $\tau$-invariant,  in $H^4(V,\ZZ)$ one has $[\Pi]= \alpha([S_1]+[S_2])$ for some $\alpha\in\ZZ$, see  Proposition~\ref{prop:prelim-scrolls}(e). Therefore, $\deg(\Pi)=6\alpha$, where $\alpha=\Pi\cdot S_i$, $i=1,2$. \footnote{Given two algebraic cycles $X,Y$ on a smooth variety $W$, we write $X\cdot Y$ for the scheme-theoretic intersection of $X$ and $Y$, while $X\cap Y$ stands for the reduced intersection of subvarieties. By abuse of notation, in the case where $X$ and $Y$ are cycles of complementary dimensions, $X\cdot Y$ can also stand for the intersection number;  in such cases, the meaning is clear from the context.}
For a general cubic scroll $S$ on $V$
the associated divisor $\Sigma(S)$ on $\Sigma$ (see Proposition \ref{prop:prelim-scrolls}(g)) is the pull-back of a general line in one of the factors $\PP^2$ under the corresponding projection. 
It follows that $C$ meets $\Sigma(S)$ transversally  in two distinct points. The latter means that among the rulings of $\Pi$ exactly two  meet
$S$. Each of them meets $S$ transversally in a single point, see Proposition \ref{prop:prelim-scrolls}(h). Consequently, $\alpha=\Pi\cdot S=2$, and hence $\deg(\Pi)=12$. 

(iv) By construction, each ruling of $\Pi$ is $\tau$-invariant and each $\tau$-invariant line on $V$ is a ruling of $\Pi$.
\end{proof}

\begin{lem} \label{lem:containment}
The fixed point set $V^\tau$ is contained in $\Pi$. 
No point of $V^\tau$ is a vertex of a cubic cone on $V$.
\end{lem}

\begin{proof} Assume on the contrary that $v\in V^\tau$ is a vertex of a cubic cone $S$ on $V$. Then
$\tau(S)$ is a cubic cone with the same vertex as $S$, hence $\tau(S)=S$ by virtue of Proposition~\ref{prop:prelim-cones}(i). However, this contradicts Proposition~\ref{prop:prelim-scrolls}(d) which says that $S$ and $\tau(S)$ belong to different components $\SSS_i$ of the Hilbert scheme $\SSS$ of cubic scrolls on $V$. 

Thus, any $v\in V^\tau$ is different from the vertices of the cubic cones on $V$. By Proposition~\ref{prop:prelim-cones}(j) there is at least one and at most three lines on $V$~passing through $v$. At least one of these lines 
is invariant under $\tau$. Indeed, this is clear in the case where the number of lines passing through $v$ is odd. Recall that $V$ is a hyperplane section of the adjoint fivefold variety $\Omega$, and the lines on $\Omega$ passing through a point $P$ form a cubic cone $S_P$ on $\Omega$. If there are exactly two lines, say $l$ and $l'$ on $V$ passing through $v$, then exactly one of these lines, say $l$ is multiple, meaning that $V=h^\bot\cap\Omega$ is tangent along $l$ to $S_P$. So $l$ and $l'$ are $\tau$-invariant, and hence $v\in l\subset \Pi$.
\end{proof}

\begin{lem} \label{lem:no-line} 
No line on $V$ is contained in $V^\tau$. 
\end{lem}

\begin{proof} Suppose a line $l$ on $V$ is pointwise fixed by $\tau$. Then 
$l$ is a ruling of $\Pi$, see Proposition~\ref{lem:scroll}(iv). 
If  $\tilde l$ is a ruling of $\tilde\Pi$ with $s(\tilde l)=l$ then $\tau$ 
acting on $\tilde\Pi$  fixes $\tilde l$ pointwise 
and preserves each ruling of $\tilde\Pi$, 
see again Proposition~\ref{lem:scroll}(iv).  
Let $U\cong\A^1\times \PP^1$ be a trivialized ruled open neighborhood of $\tilde l$ in $\tilde\Pi$ such that $\tilde l=\{0\}\times\PP^1$. 
The induced action of $\tau$ on $\A^1\times \PP^1$ is trivial on the factor $\A^1$ and defines a morphism $\phi\colon\A^1\to\PGL(2,\CC)$ such that $\phi(0)=1$ and $\phi^2(t)=1$ 
for all $t\in\A^1$. By the monodromy theorem, $\phi$ lifts to a morphism $\tilde\phi\colon\A^1\to \SL(2,\CC)$ such that $\tilde\phi(0)=1$ and $\tilde\phi^2(t)\in\{\pm 1\}$ for all $t\in\A^1$. Hence, $\tilde\phi^2(t)=1$  for all $t\in\A^1$. Since ${\rm trace}(\tilde\phi(t))\in\{0,\pm 2\}$ for any $t\in\A^1$ and 
${\rm trace}(\tilde\phi(0))=2$ one has ${\rm trace}(\tilde\phi(t))=2$  for all $t\in\A^1$. Therefore, $\tilde\phi(t)=1$ for all $t\in\A^1$. It follows that $\tau|_{\tilde\Pi}={\rm id}_{\tilde\Pi}$, and so $\tau|_{\Pi}={\rm id}_{\Pi}$.
 
By~Propositions~\ref{prop:prelim-cones}(k) 
and Theorem~\ref{thm:discrete}(ii)
there exists on $V$ a pair  $(S_0,\tau(S_0))$ of disjoint $\Aut^0(V)$-invariant cubic cones. Then a general cubic scroll $S\in\SSS_1$ and its $\tau$-dual scroll $\tau(S)\in\SSS_2$ are disjoint as well. Each of them intersects $\Pi$, see the proof of Proposition~\ref{lem:scroll}(iii). Since $\tau|_{\Pi}={\rm id}_{\Pi}$, any point $x\in S\cap\Pi$ is fixed by $\tau$, hence it is a common point of $S$ and $\tau(S)$. The latter contradicts the fact that $S$ and $\tau(S)$ are disjoint. 
\end{proof}

\begin{cor}\label{cor:birational}
$s|_{\tilde\Pi}\colon\tilde\Pi\to\Pi$ is a birational morphism.
\end{cor}

\begin{proof}
We have to show that $s|_{\tilde\Pi}\colon\tilde\Pi\to\Pi$ has degree 1. Notice that two distinct points of the curve $C\subset\Sigma$ correspond to distinct lines on $V$. Hence, two distinct rulings of $\tilde\Pi$ project under $s$ into two distinct rulings of $\Pi$. Assume on the contrary that $\deg (s|_{\tilde\Pi})=m\ge 2$, that is, through a general point $P$ of $\Pi$ pass $m\ge 2$ distinct rulings of $\Pi$. Then $P$ is fixed by $\tau$ and therefore, $\tau|_{\Pi}=\id_{\Pi}$. The latter contradicts Lemma~\ref{lem:no-line}. 
\end{proof}

\begin{lem}\label{lem:Euler} The Euler characteristic $e(V^\tau)$ 
equals $0$ modulo $4$. 
\end{lem}

\begin{proof} Given a 
periodic diffeomorphism $f$ 
of a compact differentiable manifold $M$ of dimension $n$,
the Euler characteristic of the fixed point set $M^f$ equals the Lefschetz number 
\[L(f, M)=\sum_{i=0}^n (-1)^i{\rm trace}(f^*|\,{H^i(V;\CC)}),\] 
see, e.g., \cite[Proposition~5.3.11]{tD79}.
Notice that $H^1(V,\ZZ)=0$, $H^2(V,\ZZ)=\Pic(V)=\ZZ$ and $\tau^*$ acts trivially on $\Pic(V)$. 
Furthermore, from the decomposition $H^4(V,\ZZ)=\ZZ [S]\oplus\ZZ [\tau(S)]$, 
see Proposition~\ref{prop:prelim-scrolls}(e), 
we deduce that
${\rm trace}(\tau^*|\,H^4(V,\CC))=0$. Thus, we have 
\begin{equation}\label{eq:even} 
\sum_{k=0}^4 {\rm trace}(f^*|\,{H^{2k}(V;\CC)})=4.
\end{equation}
The components of the Hodge decomposition
\[H^{2k+1}(V,\CC)=\bigoplus_{p+q=2k+1,\,p<q} (H^{p,q}(V,\CC)\oplus \overline{H^{p,q}(V,\CC)})\]
are $\tau$-invariant and
\[{\rm trace}(\tau^*|\,H^{p,q}(V;\CC))={\rm trace}(\tau^*|\,\overline{H^{p,q}(V,\CC)}).\]
Therefore, 
\begin{equation}\label{eq:odd-1} 
{\rm trace}(\tau^*|\,{H^{2k+1}(V;\CC)}) \equiv 0 \mod 2.
\end{equation}
Notice that $\tau$ acts trivially on $H^2(V,\CC)$, hence it preserves the K\"{a}hler class $[\omega]\in H^{1,1}(V,\CC)$.
Furthermore, the Lefschetz isomorphism (see, e.g., \cite[Partie I, Subsection 6.23]{Dem96})
\[L^{4-i}\colon H^{i}(V;\CC)\stackrel{\cong}{\longrightarrow} H^{8-i}(V;\CC),\quad H^{p,q}(V;\CC)\stackrel{\cong}{\longrightarrow} H^{4-q,4-p}(V;\CC),\quad p+q=i\le 4,\] is 
$\tau^*$-equivariant.
It follows that
\[{\rm trace}(\tau^*|\,H^{i}(V;\CC))={\rm trace}(\tau^*|\,H^{8-i}(V;\CC)).\]
In particular, 
\[{\rm trace}(\tau^*|\,H^{3}(V;\CC))={\rm trace}(\tau^*|\,H^{5}(V;\CC)).\]
Using \eqref{eq:odd-1} we deduce the congruence
\begin{equation}\label{eq:odd-2} 
\sum_{k}{\rm trace}(\tau^*|\,{H^{2k+1}(V;\CC)})=2{\rm trace}(\tau^*|\,{H^{3}(V;\CC)}) \equiv 0 \mod 4.
\end{equation}
Finally, \eqref{eq:even} and \eqref{eq:odd-2} yield $e(V^\tau)\equiv 0 \mod 4$. 
\end{proof}

The next proposition is the main result of this subsection.

\begin{prop}\label{prop:V-tau} $\,$
\begin{enumerate}
\item[(i)]
The fixed point set $V^\tau=\Pi^\tau$ is the union of two disjoint smooth rational curves $E_1, E_2$ on $\Pi$. Each of them meets any ruling of $\Pi$ in a unique point. Through any point of $\Pi\setminus (E_1\cup E_2)$ passes a unique ruling of $\Pi$.
\item[(ii)] The morphism $s|_{\tilde\Pi}\colon\tilde\Pi\to\Pi$ is bijective over $\Pi\setminus (E_1\cup E_2)$.  
\end{enumerate}
\end{prop}

\begin{proof} (i)
Recall that $\rho|_{\tilde\Pi}\colon\tilde\Pi\to C$ is the projection of a smooth rational surface scroll. 
By Lemma~\ref{lem:no-line} any ruling $\tilde r$ of $\tilde\Pi$ is $\tau$-invariant and $\tau|_{\tilde r}$ is not identical. 
Hence $\tau|_{\tilde r}$ has exactly two fixed points. So, the projection 
$\rho|_{\tilde\Pi^\tau}\colon\tilde\Pi^\tau\to C\cong\PP^1$
is two-sheeted and unramified. It follows that $\tilde\Pi^\tau$ is a disjoint union 
of two sections, say $\tilde E_1$ and $\tilde E_2$ of the projection 
$\rho|_{\tilde\Pi}\colon\tilde\Pi\to C$, where $\tilde E_1$ and $\tilde E_2$ are smooth rational curves. 

We have $V^\tau=\Pi^\tau=s(\tilde\Pi^\tau)=E_1\cup E_2$ where $E_i=s(\tilde E_i)$, see Lemma~\ref{lem:containment} and Corollary \ref{cor:birational}.  The fixed point set of a reductive group action on a smooth variety  is smooth, so $V^\tau=E_1\cup E_2$ is smooth. Hence either 
$E_1\neq E_2$, $E_1\cap E_2=\emptyset$ and $E_i$ is a smooth rational curve for $i=1,2$, or $V^\tau=E_1=E_2$ is a smooth rational curve, say $E$. 
However, the latter contradicts 
Lemma~\ref{lem:Euler}, since $e(E)=2\not\equiv 0\mod 4$. \footnote{Alternatively, under the assumption that $E_1=E_2$
it can be shown that
$s|_{\tilde \Pi}\colon \tilde \Pi\to \Pi$ sends some pair $(p_1,p_2)$ of distinct points 
$p_i\in \tilde E_i\cap\tilde r$, $i=1,2$
on the same ruling $\tilde r$ of $\tilde \Pi$ to the same point of $E$. Indeed, 
$\iota\colon\tilde r\mapsto (p_1,p_2)$ embeds $C$ in $E_1\times E_2$, and $s|_{E_1\cup E_2}$ induces a morphism $\phi\colon E_1\times E_2\to E\times E$. The curve $\phi\circ\iota(C)$ on $E\times E$ intersects the diagonal $\delta$ of $E\times E$, which proves our claim. 
However,  
we know already that $r=s(\tilde r)$ is a ruling of $\Pi$ and $s|_{\tilde r}\colon \tilde r\to r$ 
is an isomorphism. This yields the desired contradiction.}

Since $s$ maps  any ruling of $\tilde\Pi$ isomorphically to a ruling of $\Pi$ 
and maps the section $\tilde E_i$ of $\tilde\Pi\to C$ onto $E_i$ for $i=1,2$, we see that $E_i$ intersects any ruling of $\Pi$ at a single point.

Suppose that  through a point $p\in \Pi\setminus (E_1\cup E_2)=\Pi\setminus\Pi^\tau$ pass two rulings of $\Pi$. These rulings are $\tau$-invariant, and therefore they also pass through the point $\tau(p)\neq p$, a contradiction. Thus, any multi-branched point of $\Pi$ is contained in $E_1\cup E_2$. This proves (i). 

(ii) 
Since $s|_{\tilde\Pi}\colon \tilde\Pi\to\Pi$ restricted to any ruling of $\tilde\Pi$ is bijective  and by (i) no two rulings of $\Pi$ intersect outside $E_1\cup E_2$, it follows that $s|_{\tilde\Pi}\colon\tilde\Pi\to\Pi$ is bijective over $\Pi\setminus (E_1\cup E_2)$.  
\end{proof}

\subsection{Winding families of scrolls} 
Any line $l$ on $V$ is a common ruling of a unique pair of cubic scrolls, see Proposition~\ref{prop:prelim-scrolls}(f). 
For a ruling $l$  of $\Pi$ the corresponding pair has the form $(S_l,\tau(S_l))$ where $S_l\in\SSS_1$ and 
$\tau(S_l)\in\SSS_2$. 
In a sense, the family of cubic scrolls $(S_l)_{l\in C}$ wraps (or envelops) $\Pi$, and $(\tau(S_l))_{l\in C}$ is a second such family; cf.\ Corollary \ref{lem:isomorphism}(i).  
Consider the  smooth conic $C_1={\rm pr}_1(C)$ in $\SSS_1$, see the proof of Proposition \ref{lem:scroll}(i). 
There is also a two-parameter family of $\tau$-invariant sextic scrolls $(D_S)_{S\in\SSS_1\setminus C_1}$ 
wrapping around $\Pi$,  see Proposition \ref{lem:disj-scrolls}(iii) and Corollary \ref{lem:isomorphism}(ii). 
In Section \ref{sec:smooth} we will use the latter family in order to show that $\Pi$ is smooth. 
 
By~Proposition~\ref{prop:prelim-cones}(k)  there exists on $V$ a pair  $(S,\tau(S))$ of disjoint $\Aut^0(V)$-invariant cubic cones. 
To every such pair there corresponds a unique $\Aut^0(V)$-invariant rational sextic scroll, see  \cite[Lemma~2.2]{PZ23}. More generally, we have the following facts. 

\begin{prop}\label{lem:disj-scrolls} 
Assume that a cubic scroll $S$ on $V$ contains no ruling of $\Pi$ \footnote{By  Lemma~\ref{lem:containment} any cubic cone $S$ on $V$ verifies  this assumption. Hence, so does as well the general cubic scroll on $V$; see also Corollary \ref{lem:isomorphism} for a more precise statement.}. Then the following hold.
\begin{enumerate}\item[(i)] $S$ and $\tau(S)$ are disjoint.
\item[(ii)]  Let $A_S$ be as in Proposition~\ref{prop:prelim-scrolls}(h). 
Then $\Gamma_S:=S\cap\tau(A_S)$ is a rational twisted cubic curve on $S$. 
\item[(iii)]  The twisted cubic curves $\Gamma_S$ and $\tau(\Gamma_S)$ 
are disjoint sections of a 
rational sextic scroll $D_S$ on $V$ such that:
\begin{itemize} 
\item any two distinct rulings of $D_S$ are disjoint;
\item any line on $V$ meeting both $\Gamma_S$ and $\tau(\Gamma_S)$ is a ruling of $D_S$ and vice versa; 
\item $D_S$ is the join of $\Gamma_S$ and $\tau(\Gamma_S)$;
\item
$D_S$ is the image under a bijective morphism of a rational normal scroll $D_S^{\rm norm}$ in $\PP^{7}$. The latter scroll is the image of
$\PP^1\times\PP^1$ embedded in $\PP^{7}$ by a complete linear system of bidegree $(1,3)$. 
\end{itemize}
\item[(iv)] The scroll $D_S$ is $\tau$-invariant and carries exactly 2 distinct $\tau$-invariant rulings $r_1,r_2$ and exactly 
$4$ fixed points of $\tau$. Namely, the intersection
\[V^\tau\cap D_S=\Pi^\tau\cap D_S=(E_1\cup E_2)\cap (r_1\cup r_2)\] 
consists of the $\tau$-fixed points $p_{i,j}=E_i\cdot r_j$, $i,j=1,2$.
\item[(v)] Every line on $A_{S}\cap\tau(A_{S})$ is a ruling of $D_S$. 
\end{enumerate}
\end{prop}

\begin{proof} (i) Suppose $S\cap\tau(S)\neq\emptyset$. Since $S$ and $\tau(S)$ 
belong to different components of the Hilbert scheme $\SSS$, they share a unique 
common ruling $r$, 
see ~Proposition~\ref{prop:prelim-scrolls}(f). Being unique, $r$ is $\tau$-invariant, 
hence $r$ is a common ruling of $\Pi$ and $S$ contrary to our assumption.
Thus, $S\cap\tau(S)=\emptyset$, as stated in (i). 

(ii) The hyperplane section $A_{\tau(S)}=\tau(A_S)$ intersects each ruling of $S$ at a single point. 
Indeed, assuming the opposite,  $A_{\tau(S)}$ contains a ruling, say $r$ of $S$. 
Then $r$ meets $\tau(S)$, see  Proposition~\ref{prop:prelim-scrolls}(h). Thus, $S\cap\tau(S)\neq\emptyset$.
However, the latter contradicts (i).  

Notice that if $S$ is a cubic cone, then $A_{\tau(S)}$ cannot pass through the vertex $v$ of $S$. Indeed, otherwise the unique line in $A_{\tau(S)}$ through $v$ that intersects $\tau(S)$ is a ruling of $S$, which contradicts (i). 
Finally, $\Gamma_S=S\cap A_{\tau(S)}$ is a section of the ruling on $S$, and hence this is a smooth, irreducible rational cubic curve. 

(iii) \footnote{Cf. also \cite[Lemma~2.11(ii)--(iii)]{PZ23} for a particular case where $S$ is a cubic cone. } By  Proposition~\ref{prop:prelim-scrolls}(h) 
for any $P\in\Gamma_S=S\cap A_{\tau(S)}$ there is a unique line $l_P$ on $A_{\tau(S)}$ through $P$ which meets $\tau(S)$ at a unique point, say $P'$. Clearly, $l_P=l_{P'}$ where $l_{P'}$ is a unique line on $A_S$ through $P'$ which meets $S$. It follows that 
$l_P\subset A_S\cap A_{\tau(S)}$. 
The family of lines $\{l_P\}_{P\in\Gamma_S}$ sweeps a scroll, say $D_S$, which contains two disjoint smooth sections $\Gamma_S$ and $\Gamma_{\tau(S)}$ and is contained in $A_S\cap A_{\tau(S)}$. By our construction any line in $V$ joining $\Gamma_S$ and $\Gamma_{\tau(S)}$ is a ruling of $D_S$, and no two distinct rulings meet. 
In particular, $D_S=J(\Gamma_S,\Gamma_{\tau(S)})$ is the join of $\Gamma_S$ and $\Gamma_{\tau(S)}$.
Hence, $D_S$ is a rational scroll of degree $\deg(D_S)=\deg(\Gamma_S)+\deg(\Gamma_{\tau(S)})=6$,
see \cite[Examples~8.17 and~19.5]{Har92}.

Let us show the last statement of (iii).  Notice that the normalization $D_S^{\rm norm}$ of $D_S$ is isomorphic to a Hirzebruch surface $\FF_e$ for some $e\ge 0$.  Since distinct rulings of $D_S$ are disjoint, the normalization morphism $\nu\colon D_S^{\rm norm}\to D_S$ is a bijection. It is $\tau$-equivariant with respect to the natural $\tau$-action on the normalization. 

Suppose on the contrary that $e>0$. Let $E_0$ be the exceptional section of $D_S^{\rm norm}$ and $F$ be a ruling of $D_S^{\rm norm}$. Let $\tilde\Gamma_S$ be the proper transform of $\Gamma_S$ on $D_S^{\rm norm}$; this is a section of the induced ruling on $D_S^{\rm norm}\cong\FF_e$. One has 
\[\tilde\Gamma_S\sim E_0+\alpha F\sim\tau(\tilde\Gamma_S)\quad\text{for some}\quad \alpha\ge e>0.\]  
For the intersection index of curves on $D_S^{\rm norm}$  we have
\[0=\tilde\Gamma_S\cdot \tau(\tilde\Gamma_S)=
\tilde\Gamma_S^2=2\alpha-e,\quad\text{hence}\quad 0<e=2\alpha\le\alpha,\]
a contradiction. Therefore, $e=0$, that is, $D_S^{\rm norm}\cong\PP^1\times\PP^1$. 

Assume that the fiber $F$ of ${\rm pr}_1\colon\PP^1\times\PP^1\to\PP^1$ is sent to a ruling of $D_S$ under $\nu$.  Since $\alpha=0$ for $e=0$ in the above calculation, we have $\tilde\Gamma_S\sim E_0$ is a constant section of
${\rm pr}_1$ under ${\rm pr}_2\colon\PP^1\times\PP^1\to\PP^1$. Let $H$ be a general hyperplane section of $D$ and $\tilde H$ be the proper transform of $H$ on $D_S^{\rm norm}$. 
Since $\tilde H\cdot F=1$ and $\tilde H\cdot \tilde\Gamma_S=H\cdot\Gamma_S=3$ we have $\tilde H
\sim E_0+3F$. 

The complete linear system $|E_0+3F|$ embeds $D_S^{\rm norm}$ onto a rational normal sextic scroll in $\PP^7$ denoted by the same symbol. The images of $\tilde\Gamma_S$ and $\tau(\tilde\Gamma_S)$ are two cubic curves whose linear envelops are two skew $\PP^3$'s in $\PP^7$. 
The scroll $D_S^{\rm norm}$ is the join of these cubic curves, see again \cite[Examples~8.17 and~19.5]{Har92}. The scroll $D_S$ is the image of $D_S^{\rm norm}=\PP^1\times\PP^1$ under a linear map given by a subsystem of
the  linear system $|E_0+3F|$.
This completes the proof of  (iii). 

(iv) The $\tau$-invariance of $D_S$ in (iv) follows from the facts that the pair $(S,\tau(S))$
is $\tau$-invariant and $D_S$ is uniquely defined by  $S$ due to (iii). We claim that the rulings of $D_S$ form a unique pencil of lines on $D_S$. Indeed, if a surface $X$ in $\PP^n$ contains a two-parameter family of lines, then $X=\PP^2$ is a linear subspace in $\PP^n$. If the Fano variety of lines on $X$ 
is one-dimensional and has at least two one-dimensional components, then $X\subset\PP^n$ is a smooth quadric surface, see 
\cite[Lemma A.1.1 and its proof]{KPS18}. Since $\deg(D_S)=6$, see (iii), this proves our claim. 

It follows that the pencil of rulings on $D_S$ is $\tau$-invariant. The induced action of $\tau$ on this pencil is not identical, since otherwise any ruling of $D_S$ is $\tau$-invariant, and therefore $D_S=\Pi$.
However, the latter is impossible because $\deg(D_S)=6$ by (iii) and $\deg(\Pi)=12$, see Proposition~\ref{lem:scroll}. 

Thus, the action of $\tau$ on $\PP^1$ parameterizing the pencil has exactly two fixed points, that is, there are exactly two  $\tau$-invariant rulings of $D_S$, say $r_1$ and $r_2$. By Lemma~\ref{lem:no-line} 
the $\tau$-action on $r_i$ is not identical. So, $r_i$ contains exactly two fixed points of $\tau$ for $i=1,2$. If  a ruling $r$ of $D_S$ different from $r_1$ and $r_2$ carries a fixed point $x$ of $\tau$ then $\tau(r)$ is as well a ruling of $D_S$ passing through $x$. By (iii) we have $r=\tau(r)$. Hence, $r$ is a third $\tau$-invariant ruling of $D_S$, a contradiction. Therefore, the points in $(r_1\cup r_2) \cap (E_1\cup E_2)$ are the only $\tau$-fixed points on $D_S$. 

(v) Any line $l$ on $A_S\cdot\tau(A_S)$ meets $S$ and  $\tau(S)$ in points of $\Gamma_S$ and $\tau(\Gamma_S)$, respectively, see Proposition \ref{prop:prelim-scrolls}(h). Due to (iii)  $l$ is a ruling of $D_S$.
\end{proof}

\begin{rems} 1. Assume that $S$ and $\tau(S)$ are disjoint. One can show that the codimension 2 linear section $A_{S}\cap\tau(A_{S})$ of $V$ is a union of three sextic surfaces $D_S$, $R_S$ and $\tau(R_S)$. In the case that $\Aut^0(V)\cong\GL(2,\CC)$, and only in this case,  one has $R_S=D_S$, so that $A_S\cdot\tau(A_S)=3D_S$, cf.~\cite[Theorem~13.5(c)]{PZ18}. 

2. In fact, the scroll $D_S$ is smooth and the normalization morphism $D_S^{\rm norm}\to D_S$ is isomorphism. This can be shown following the lines of the proof of Proposition \ref{prop:smoothness} and replacing the pair $(E_1,E_2)$ by the one $(\Gamma_S, \tau(\Gamma_S))$. However,  in what follows we do not use this fact. 
\end{rems}

Recall that the Hilbert scheme $C$ of rulings of $\Pi$ is a smooth rational quartic curve in $\Sigma$ contained in the diagonal of $\SSS_1\times\SSS_2\cong\PP^2\times(\PP^2)^\vee$. The projection $C_i={\rm pr}_i(C)\subset\SSS_i$ is a smooth conic on $\SSS_i\cong\PP^2$, $i=1,2$, see Proposition~\ref{lem:scroll}(i) and its proof. Since $C=\Sigma^\tau$ is $\tau$-invariant and $\tau(\SSS_1)=\SSS_2$, we have $\tau(C_1)=C_2$, see Subsection \ref{ss:first} and Corollary \ref{cor:aut-GL2}.

\begin{cor}\label{lem:isomorphism} $\,$
\begin{enumerate}
\item[(i)]
A pair of cubic scrolls $(S,\tau(S))\in\SSS_1\times\SSS_2$ shares a unique common ruling $r=r(S)$ if and only if $S\in C_1$. In the latter case $r$ is a ruling of $\Pi$. 
\item[(ii)] Let $S\in\SSS_1\setminus C_1$ and let $D_S$ be
the  
rational sextic scroll  as in Proposition~\ref{lem:disj-scrolls}(iii). Then $D_S\cap \Pi=r_1(S)\cup r_2(S)$ is a union of two distinct rulings. 
\item[(iii)] Let $C^{(2)}$ stand for the symmetric square of $C$ and $\delta$ stand for the diagonal of $C^{(2)}$. The correspondence 
\begin{equation}\label{eq:isom} \SSS_1\setminus C_1\ni S\mapsto \{r_1(S), r_2(S)\}\in C^{(2)}\setminus\delta\cong \PP^2\setminus\,\{\text{a smooth conic}\}\end{equation}
is an isomorphism.
\end{enumerate}
\end{cor}

\begin{proof} (i)
By virtue of Proposition~\ref{prop:prelim-scrolls}(f) $S$ and $\tau(S)$ intersect if and only if they share a unique common ruling. This  ruling, say $r$, is $\tau$-invariant, hence is a ruling of $\Pi$, see Proposition~\ref{lem:scroll}. Thus, $r\in C$, and so $S={\rm pr}_1(r)\in C_1$, see again Proposition~\ref{prop:prelim-scrolls}(f). The converse is immediate from Proposition~\ref{lem:disj-scrolls}(i).

Statement (ii) follows from (i) due to Proposition~\ref{lem:disj-scrolls}(iv) and its proof. 

To show (iii)
recall that the lines on $V$ meeting $S$ sweep out a hyperplane section $A_S$ of $V$ singular along $S$, see Proposition \ref{prop:prelim-scrolls}(h). The Hilbert scheme $\Sigma(S)$ of these lines is the preimage ${\rm pr}_2^{-1}(h_S)$ where $h_S\subset\SSS_2=(\SSS_1)^\vee$ is the dual line of the point $S\in\SSS_1$, see Proposition~\ref{prop:prelim-scrolls}(g). 
The smooth conic
\[\sigma_S:=\Sigma(S)\cap \Sigma(\tau(S))={\rm pr}_2^{-1}(h_S)\cap{\rm pr}_1^{-1}(h_{\tau(S)})\] 
parameterizes the lines on $V$ which meet both $S$ and $\tau(S)$, that is, the lines on $A_S\cap A_{\tau(S)}$. By Proposition~\ref{lem:disj-scrolls}(v) these are the rulings of $D_S$. 

The line $h_S\subset \SSS_2$ meets the smooth conic $C_2={\rm pr}_2(C)$ in an unordered pair of points. Hence $\Sigma(S)={\rm pr}_2^{-1}(h_S)$ meets $C$ in an unordered pair of points $\{r_1(S),r_2(S)\}$. 
The corresponding rulings $r_1(S),r_2(S)$ of $\Pi$ intersect $S$. Since these rulings are $\tau$-invariant and so is the scroll $D_S$,
these are the common rulings of $\Pi$ and $D_S$. By (i) $r_1(S)=r_2(S)$ if and only if $S\in C_1$. 
This yields the morphism in \eqref{eq:isom}. 

To show that the correspondence  in  \eqref{eq:isom} is one-to-one, choose $\{r_1,r_2\}\in C^{(2)}\setminus\delta=\PP^2\setminus \{\text{a smooth conic}\}$. Let $l$ be the line on $\SSS_2\cong\PP^2$ passing through the points ${\rm pr}_2(r_1)$ and ${\rm pr}_2(r_2)$ if these are distinct and tangent to $C_2$ at ${\rm pr}_2(r_1)$ otherwise. In the former case, the point $S=l^\vee$ in the dual projective plane $\SSS_1$  corresponds to a unique cubic scroll such that $D_S$ and $\Pi$ share the common rulings $r_1$ and $r_2$. Thus, the morphism in \eqref{eq:isom} is a bijection, so an isomorphism. 
\end{proof}

\begin{prop}\label{prop:bijection} $\,$
\begin{enumerate}
\item[(i)] No two distinct rulings of $\Pi$ meet. 
\item[(ii)] $s|_{\tilde\Pi}\colon\tilde\Pi\to\Pi$ 
is a bijective normalization morphism. 
\item[(iii)] $\deg(E_1)+\deg(E_2)=12$.
\item[(iv)] Any line on $V$ that intersects both $E_1$ and $E_2$ is a ruling of $\Pi$.
\end{enumerate}
\end{prop}

\begin{proof} (i) 
The unordered pairs of distinct rulings of $\Pi$ are in bijection with the points of $C^{(2)}\setminus\delta$, see \eqref{eq:isom}.  
Every such pair corresponds to a pair of rulings $\{r_1,r_2\}$ of a
sextic scroll $D_S$ for some $S\in\SSS_1\setminus C_1$, see Corollary~\ref{lem:isomorphism}(iii). However, 
no two distinct rulings of $D_S$ meet, see Proposition \ref{lem:disj-scrolls}(iii). 

(ii) Since $\tilde\Pi$ is smooth and $s|_{\tilde\Pi}\colon\tilde\Pi\to\Pi$ is bijective due to (i), the latter is a normalization morphism. 

 (iii) Recall that $\deg(\Pi)=12$ and that $\tilde E_i$ and $E_i$ are smooth rational curves, see Propositions \ref{lem:scroll}(iii) and \ref{prop:V-tau}(i). Hence by (ii), $s|_{\tilde E_i}\colon \tilde E_i\to E_i$ is an isomorphism. Now (iii) follows from the classical Edge formula \cite[Section~I, p.~17]{Edge31}, see \cite[Example~19.5]{Har92} for a modern treatment.
 
 (iv) Let $l$ be a line on $V$ which meets $E_i$ in a point $P_i$, $i=1,2$. Since $\tau(P_i)=P_i$, we have $P_i\in \tau(l)$ for $i=1,2$, and therefore $\tau(l)=l$. Now the result follows from Proposition \ref{lem:scroll}(iv).
\end{proof}

\subsection{Numerical data of $\Pi$}

\begin{lem}\label{lem:3-components} Given a cubic scroll  $S\in\SSS_1$, let $A_S$ be 
the hyperplane section of $V$ singular along $S$, see Proposition~\ref{prop:prelim-scrolls}(h). 
Then for the $1$-cycle $A_S\cdot \Pi$ on $\Pi$ we have  
\begin{equation}
\label{eq:Gamma} A_S\cdot \Pi = n_1r_1+n_2r_2+\gamma_S\end{equation}
where $n_i=n_i(S)\ge 1$, $\gamma_S$ is a section of $\Pi$  and  the  rulings $r_1$ and $r_2$ of $\Pi$ are disjoint if $S\in\SSS_1\setminus C_1$ and equal if $S\in C_1$.
\end{lem}

\begin{proof} By Corollary \ref{lem:isomorphism} $S$ and $\tau(S)$ are disjoint if $S\in\SSS_1\setminus C_1$ and share a unique common ruling $r$ otherwise; in the latter case $r$ is a ruling of $\Pi$.  
In the former case the corresponding sextic scroll $D_S\subset A_S\cap\tau(A_S)$ shares  with $\Pi$ precisely two common rulings $r_1$ and $r_2$, see Proposition~\ref{lem:disj-scrolls}(iv). These rulings $r_1$ and $r_2$ participate in the $1$-cycle $\Upsilon_S=A_S\cdot \Pi$ with positive multiplicities $n_1$ and $n_2$, respectively. No third ruling $r$ of $\Pi$ is contained in $A_S$. Indeed, otherwise $r=\tau(r)\subset A_S\cap\tau(A_S)$ (see Proposition \ref{lem:scroll}) would be a third common ruling of $D_S$ and $\Pi$, which contradicts Corollary \ref{lem:isomorphism}(ii). 
Clearly, the residual effective 1-cycle $\gamma_S$ in~\eqref{eq:Gamma} is reduced  
and a section of $\Pi$. 
\end{proof}

\begin{lem}\label{lem:winding family} 
For the general cubic scroll $S\in\SSS_1$ one has $n_1=n_2$ in~\eqref{eq:Gamma}.
\end{lem}
 
\begin{proof} 
By Corollary~\ref{lem:isomorphism} the rulings $r_1(S)$ and $r_2(S)$ are distinct for $S\in\SSS_1\setminus C_1$. The unordered pair of coefficients $\{n_1(S), n_2(S)\}$ in~\eqref{eq:Gamma} is constant for $S$ from a suitable open subset $U\subset\SSS_1\setminus C_1$. Identifying $\SSS_1\setminus C_1$ with $C^{(2)}\setminus\delta=\PP^2\setminus\delta$ via~\eqref{eq:isom} consider the 2-sheeted covering $(\PP^1\times\PP^1)\setminus \tilde\delta\to\PP^2\setminus\delta$, where $\tilde\delta$ stands for the diagonal in $\PP^1\times\PP^1$. The monodromy of the covering restricted to $U$ interchanges the members of the ordered pair of points from $(\PP^1\times\PP^1)\setminus \tilde\delta$. It follows that $n_1(S)=n_2(S)$ for $S\in U$. 
\end{proof}

Recall that according to Proposition~\ref{prop:bijection}, $\tilde\Pi$ is isomorphic to a Hirzebruch surface $\FF_e$ for some $e\ge 0$.

\begin{lem}\label{lem:exclusion} In the notation of Lemma \ref{lem:3-components}
the following conditions are equivalent:
\begin{itemize}
\item[\rm ($*$)]  $e=0$, that is, $\tilde\Pi\cong\PP^1\times\PP^1$;
\item[\rm ($**$)]  $\deg(E_1)=\deg(E_2)=6$. 
\end{itemize}
Suppose $S\in\SSS_1$ is a cubic scroll such that $\gamma_S\neq E_i$ for $i=1,2$ \footnote{Notice that $\tilde\gamma_S\neq \tilde E_i$ for a general $S\in\SSS_1$, $i=1,2$. Indeed, otherwise 
$E_i\subset A_S\cap\tau(A_S)$ for any $S\in\SSS_1$ contrary to Proposition \ref{prop:prelim-cones}(l). }. Then conditions $(*)$ and $(**)$ are fulfilled provided $n_1=n_2=3$ in~\eqref{eq:Gamma}. 
\end{lem}

\begin{proof} 
Assume that $e>0$, and let $\tilde E_0$ be the exceptional section of $\tilde \Pi\cong\FF_e$. 
Since $\tau$ preserves each ruling of $\tilde \Pi$ and $\tau(\tilde E_0)=\tilde E_0$, one has $\tau|_{\tilde E_0}={\rm id}_{\tilde E_0}$.
So, $\tilde E_0$ is one of the components of the fixed point set $\tilde \Pi^\tau=\tilde E_1\cup\tilde E_2$; we may assume that $\tilde E_0=\tilde E_1$. 

In the sequel we use the notation
\[\tilde A_S=s^*(A_S),\quad \tilde\Upsilon_S=\tilde A_S\cdot\tilde\Pi,\quad\tilde \gamma_S=s^*(\gamma_S)\cdot\tilde\Pi\quad\text{and}\quad \tilde r_i=s^*(r_i)\cdot\tilde\Pi.\]
Thus,  $\tilde\Upsilon_S,\,\tilde \gamma_S$ and $\tilde r_i$ are the proper transforms of $\Upsilon_S,\,\gamma_S$ and $r_i$ under the bijective morphism $s|_{\tilde\Pi}\colon\tilde\Pi\to\Pi$.

For a general $e\ge 0$ we have $\tilde E_2\sim\tilde E_1+eF$ where $F$ is a ruling of $\tilde \Pi$. 
Since $(\tilde A_S|_{\tilde \Pi})^2=12$ and the linear system $|\tilde A_S|$ restricted to $\tilde \Pi$ is base point free, see Proposition \ref{prop:bijection}(ii), one has $\tilde A_S|_{\tilde \Pi}\sim\tilde E_1+aF$ where $a=6+e/2\ge e$. 
It follows that 
\begin{equation}\label{eq:deg} \deg(E_1)=\tilde A_S\cdot \tilde E_1=6-e/2\quad\text{and}\quad \deg(E_2)=\tilde A_S\cdot \tilde E_2=6+e/2,\end{equation}
which is in line with Proposition \ref{prop:bijection}(iii). 
Therefore,  $e=0$ if and only if $\deg(E_1)=\deg(E_2)=6$, that is, $(*)$ and  $(**)$ are equivalent.

Using Proposition \ref{lem:scroll}(iii) and Lemmas~\ref{lem:3-components} --\ref{lem:winding family}  and letting $n=n_1=n_2$ in~\eqref{eq:Gamma}, for the hyperplane section $\Upsilon_S= A_S \cdot \Pi$ we obtain
\[12=\deg(\Pi)=\deg(\Upsilon_S)=2n+\deg(\gamma_S).\] 
Using the Projection Formula,~Proposition \ref{prop:bijection}(iii) and~\eqref{eq:Gamma}  we deduce 
\[\tilde\Upsilon_S=n(\tilde r_1+\tilde r_2)+\tilde \gamma_S \quad\text{and}\quad
12=\tilde A_S\cdot  \tilde\Upsilon_S=2n+\deg(\gamma_S).\]
By  \eqref{eq:deg} we have $\tilde A_S\cdot (\tilde E_1+\tilde E_2)=12$. Hence, on $\tilde \Pi$ one has
\[12=\tilde \Upsilon_S\cdot (\tilde E_1+\tilde E_2)=
4n+\tilde \gamma_S\cdot (\tilde E_1+\tilde E_2).\]
It follows that $n\le 3$ provided $\gamma_S\neq E_1$, and $n=3$ if and only if $\deg \tilde \gamma_S=6$. 

Let now $S\in \SSS_1$ be such that $n_1=n_2=3$ and $\gamma_S$ in \eqref{eq:Gamma} is different from $E_1$ and $E_2$. In this case the intersection index $\tilde\gamma_S\cdot\tilde E_i$ on $\tilde\Pi$ is zero for $i=1,2$, and therefore $\tilde E_1,\,\tilde E_2$ and $\tilde \gamma_S$ are three disjoint sections of $\tilde\Pi\to C$. Hence, $e=0$, and if $\tilde \Pi\cong\PP^1\times\PP^1\stackrel{{\rm pr}_1}{\longrightarrow}\PP^1$ defines the ruled surface structure on $\tilde\Pi$, then these three sections are constant with respect to the projection $\tilde \Pi\cong\PP^1\times\PP^1\stackrel{{\rm pr}_2}{\longrightarrow}\PP^1$. Thus, both $(*)$ and $(**)$  are fulfilled.
\end{proof}

\section{Starting the proof of the main theorem}\label{sec:proof} 
The following theorem is part of Theorem \ref{thm:mthm}. 

\begin{thm}\label{thm:towards-mthm}
Let $V$ be a Fano--Mukai fourfold  of genus $10$ half-anticanonically embedded in $\PP^{12}$, and let $\tau\in \Aut(V)\setminus\Aut^0(V)$ be an involution. Then the fixed point set $V^\tau$ is a union of two disjoint smooth rational sextic curves $E_1$ and $E_2$. Furthermore, 
there is a surface scroll $\Pi=\Pi(V,\tau)$ in $V$ verifying the following.
\begin{enumerate}\item[(i)] Each ruling of $\Pi$ is $\tau$-invariant and each $\tau$-invariant line on $V$ is a ruling of $\Pi$.
\item[(ii)] $\Pi$ has degree $12$ and its normalization $\tilde\Pi$ is a rational normal scroll isomorphic to $\PP^1\times\PP^1$ and embedded in $\PP^{13}$ by the complete linear system of type $(1,6)$. The normalization morphism $\tilde\Pi\to\Pi$ is bijective.
\item[(iii)] An isomorphism $\tilde\Pi\cong\PP^1\times\PP^1$ sends the proper transforms $\tilde E_i$ of $E_i$ on $\tilde\Pi$, $i=1,2$ into constant sections of the projection ${\rm pr}_1\colon\PP^1\times\PP^1\to\PP^1$.
\end{enumerate}
\end{thm}

It suffices to show that $\deg(E_i)=6$ for $i=1,2$, the remaining assertions being already established in Propositions \ref{prop:V-tau}, \ref{prop:bijection} and Lemma \ref{lem:exclusion} or follow immediately from these.
We prove the equalities $\deg(E_i)=6$, $i=1,2$ separately for each type of the group $\Aut^0(V)$, see Propositions~\ref{prop:GL2},~\ref{prop:Gm2} and~\ref{prop:Ga-Gm}. Once we know that these equalities hold provided $\Aut^0(V)=\GL(2,\CC)$, the other two cases, where $\Aut^0(V)=\Gm^2$ and $\Aut^0(V)=\Ga\times\Gm$, respectively, are reduced to this one by a specialization argument.

\subsection{The case $\Aut^0(V)=\GL(2,\CC)$}\label{ss:GL2} 

\begin{prop}\label{prop:GL2} 
Suppose $\Aut^0(V)=\GL(2,\CC)$. Then $\deg(E_1)=\deg(E_2)=6$. 
\end{prop}

\begin{proof}
Let $S\in\SSS_1$ be an $\Aut^0(V)$-invariant cubic cone on $V$. There is a unique such cone in $\SSS_1$, and therefore $\tau(S)$ is a unique $\Aut^0(V)$-invariant cubic cone in $\SSS_2$, see Proposition \ref{prop:cubic-cones}. Furthermore, $S$ and $\tau(S)$ are disjoint, see Proposition \ref{prop:prelim-cones}(k).
According to~\cite[Theorem~13.5(c)]{PZ18} one has $A_S\cdot A_{\tau(S)}=3D_S$. Hence 
\begin{equation}\label{eq:triple-intersection} A_S\cdot A_{\tau(S)}\cdot\Pi=3D_S\cdot\Pi=3(r_1+r_2),\end{equation}
see Corollary \ref{lem:isomorphism}(ii).
 It follows by~\eqref{eq:Gamma} that $A_S\cdot\Pi=3(r_1+r_2)+\gamma_S$, that is, $n_1=n_2=3$. We claim that $\gamma_S$ is different from $E_1$ and $E_2$. Indeed,  otherwise $\tau(\gamma_S)=\gamma_S$, and therefore $\gamma_S\subset A_S\cap A_{\tau(S)}\cap\Pi=D_S\cap\Pi$. The latter contradicts \eqref{eq:triple-intersection}, which proves our claim.
Now Lemma~\ref{lem:exclusion} applies and gives the result. 
\end{proof}

\subsection{The case $\Aut^0(V)=\Gm^2$}\label{ss:Gm2}

In this case we use the Mukai presentation of the Fano--Mukai fourfolds of genus 10, see Subsection~\ref{ss:Mukai}.  

\begin{prop}\label{prop:Gm2} 
Assume $\Aut^0(V)=\Gm^2$. Then $\deg(E_1)=\deg(E_2)=6$.
\end{prop}

\begin{proof}
Let $T$  be a maximal torus of $G$ and $\fh={\Lie}(T)$ be the corresponding Cartan subalgebra of $\fg={\Lie}(G)$. The $T$-action on $\PP\fg=\PP^{13}$ fixes pointwise the projective line 
$\PP\fh$. Hence $T\subset \Aut (V(h))={\rm Stab}_G(\PP h)$ for any point $\PP h\in\PP\fh$. If $h\in\fg\setminus\{0\}$ is semisimple then $\Aut (V(h))$ contains a maximal torus of $G$.
Since the maximal tori in $G$ are conjugate, the line $\PP\fh$ intersects each semisimple $G$-orbit in $\PP\fg$ including $\Omega^{\mathrm sr}$ and does not intersect the orbit $\Omega^\aaa$; cf.~\cite[Proof of Lemma~1]{KR13} and \cite[Corollary~2.5.1]{PZ22}. 

Let $\mathscr{F}\to \PP\fg$  be the universal family of hyperplane sections of $\Omega$. The restriction 
\[\mathscr{F}_T:=\mathscr{F}|_{\PP\fh\setminus D_\ell}\to \PP\fh\setminus D_\ell\cong\PP^1\setminus \{3\,\text{points}\},\] 
see~\cite[Lemma~2.8]{PZ22},
 is a smooth one-parameter family of Fano--Mukai fourfolds  of genus 10 consisting of the smooth hyperplane sections $V(h)$ of $\Omega$ which satisfy $T\subset \Aut(V(h))$ where $\Aut(V(h))$ coincides with the stabilizer ${\rm Stab}_G(\PP h)$, see Theorem \ref{thm:Aut}. 

Let $N_G(T)$ stand for the normalizer of $T$ in $G$ and let $W=N_G(T)/T$ be the Weyl group of $G$. One has $W\cong \fD_6$ where $\fD_6$ is the dihedral group of order 12. By~\cite{AH17} there is a splitting $N_G(T)=T\rtimes W$. Let $\tau\in N_G(T)$ be the central element of order 2 in $W\subset N_G(T)$. Then $\tau$ acts on the Cartan subalgebra $\fh$ via the central symmetry $h\mapsto -h$, acts on $T$ via the inversion $t\mapsto t^{-1}$ and acts identically on the line $\PP \fh\subset\PP\fg$. For any point $\PP h\in\PP \fh$ the torus $T$ acts effectively on $V(h)$. The involution $\tau$ acts on the fivefold $\Omega$ preserving the hyperplane section $V(h)$. Thus, 
\[\Aut(V(h))=\mathrm{Stab}_G(\PP h)\supset \langle T, \tau\rangle\cong T\rtimes\ZZ/2\ZZ.\]
Hence $\tau$ acts on the total space of the family $\mathscr{F}_T\to\PP\fh\setminus D_\ell$ and acts effectively on the fiber $V(h)$ for any $\PP h\in \PP\fh\setminus D_\ell$. 

Since $\langle\tau\rangle\cong\ZZ/2\ZZ$ is a reductive group, the fixed point set $(\mathscr{F}_T)^\tau$ is a smooth subvariety of the smooth variety  $\mathscr{F}_T$. Any fiber $V(h)^\tau$ of $(\mathscr{F}_T)^\tau\to \PP\fh\setminus D_{\ell}$ is a union of two disjoint smooth rational curves $E_1(h), E_2(h)$ where $\deg (E_1(h)+E_2(h))=12$. Hence $(\mathscr{F}_T)^\tau$ is a smooth rational surface fibered over $\PP\fh\setminus D_\ell=\PP^1\setminus \{3\,\text{points}\}$ into a family of smooth reducible curves $E_1(h)\cup E_2(h)$ of degree 12. It follows that the lower semicontinuous function $\deg(E_i(h))$ on $\PP\fh\setminus D_\ell$ is constant for $i=1,2$. For a point $h_0\in \PP\fh\cap D_\s=\PP\fh\cap \Omega^{\mathrm{sr}}$ with $\Aut^0(V(h_0))\cong\GL(2,\ZZ)$ one has
\[\Aut(V(h_0))=\Aut^0(V(h_0))\rtimes\langle\tau\rangle\cong\GL(2,\ZZ)\rtimes\ZZ/2\ZZ,\] see~\cite[Theorem~A]{PZ22}.
By Proposition~\ref{prop:GL2} one has $\deg (E_1(h_0))=\deg (E_2(h_0))=6$. Therefore, $\deg (E_1(h))=\deg (E_2(h))=6$ for any $\PP h\in \PP\fh\setminus D_\ell$. 

The same equalities hold for any $\PP h'\in \PP\fg\setminus (D_\ell\cup D_\s)$, that is, for any $h'$ with $\Aut^0(V(h'))\cong\Gm^2$. Indeed, the projective line $\PP\fh$ meets any $G$-orbit $D_t\setminus D_\ell$, $t\notin\{\ell,\s\}$, see Theorem \ref{thm:orbits}(ii) and \cite[Corollary 3.5.1]{PZ22}. Hence
the $G$-orbit of $\PP h'$ contains a point $\PP h\in \PP\fh\setminus (D_\ell\cup D_\s)$. Thus, there is an element $g\in G\subset\mathrm{PGL}(14,\CC)$ which sends $\PP h$ to $\PP h'$, $V(h)$ isomorphically onto $V(h')$,  $\tau$ to an involution $\tau'\in \Aut(V(h'))\setminus\Aut^0(V(h'))$, and $E_i(h)$ to $E_i(h')$, $i=1,2$, up to switching the order. 

By Corollary~\ref{cor:aut-GL2} the choice of an involution $\tau\in\Aut(V)\setminus\Aut^0(V)$ is irrelevant for our purposes. Therefore, the above argument proves the proposition.
\end{proof}

\subsection{The case $\Aut^0(V)=\Ga\times\Gm$}\label{ss:Ga-Gm} The proof in this case is similar to the preceding one. We construct a one-parameter family  of Fano--Mukai fourfolds $\{V_t\}$ with $\Aut^0(V_t)=\Ga\times\Gm$ for $t\neq 0$,
which specializes to $V_0$ with $\Aut^0(V_0)=\GL(2,\CC)$, and then apply the same argument as before.

\begin{prop}\label{prop:Ga-Gm} 
Assume $\Aut^0(V)=\Ga\times\Gm$. Then $\deg(E_1)=\deg(E_2)=6$.
\end{prop}

\begin{proof} According to~\cite[Lemma 1 and its proof]{KR13}, see also Theorem~\ref{thm:orbits}(iii) 
there exist commuting  elements $f,n\in\fg$ such that $f$ is subregular semisimple, $n$ is nilpotent 
and $g=f+n$ is regular. 
We have $\PP f\in\Omega^{\mathrm{sr}}$ and $\PP (f+n)\in\Omega^\aaa$. It follows from the description 
of $G$-orbits in $\fg$, see Theorem~\ref{thm:orbits}, that the same inclusions hold if one replaces $n$ 
by $\lambda n$ with $\lambda\in\CC\setminus\{0\}$ and $f+n$ by $f+\lambda n\in\ff:={\rm span} (f,n)$. 

Notice that $\ff$  is a two-dimensional abelian Lie subalgebra of $\fg$.
Consider the projective line $\PP \ff\subset\PP\fg$. Let $\omega=\PP \ff\setminus\{\PP n\}$ 
be the image in $\PP\fg$ of the affine line 
$\{f+\lambda n\,|\,\lambda\in\CC\}$ and let $\omega^\aaa=\omega\setminus\{\PP f\}$. 
 Since $f$ is subregular, any regular element  of  $\ff\setminus\{\lambda n|\lambda\in\CC\}$ is mixed. Since $n$ is 
 nilpotent, one has $\PP n\in D_\ell\cap D_\s$, see Theorem~\ref{thm:orbits}(i). It follows that 
\[\omega=\PP\ff\setminus D_\ell\subset D_\s\setminus D_\ell\quad\text{and}\quad
\omega^\aaa=\PP\ff\cap\Omega^\aaa=\omega\cap\Omega^\aaa.\] 

Consider the smooth one-parameter  family $\mathscr{F}_\omega:=\mathscr{F}|_{\omega}\to\omega$ of hyperplane 
sections $V(h)$ of $\Omega$ with $\PP h\in\omega$. For $\PP h\in\omega$ we have by Theorem~\ref{thm:Aut},
\[\Aut^0(V(h))\cong\Ga\times\Gm\Leftrightarrow\PP h\in \omega^\aaa\quad\text{and}\quad 
\Aut^0(V(h))\cong\GL(2,\CC)\Leftrightarrow \PP h=\PP f\in\omega\setminus\omega^\aaa.\] 
Let 
\[L=\big({\rm Stab}_G(\PP f)\big)^0=\Aut^0(V(f))\cong\GL(2,\CC)\]
and 
\[\tilde L={\rm Stab}_G(\PP f)=\Aut(V(f))=L\rtimes\ZZ/2\ZZ.\] 
For $\lambda \neq 0$ we let $g_\lambda=f+\lambda n$,
\[H_\lambda=\big({\rm Stab}_G(\PP g_\lambda)\big)^0=\Aut^0(V(g_\lambda))\cong\Ga\times\Gm\]
and 
\[\tilde H_\lambda={\rm Stab}_G(\PP g_\lambda)=\Aut(V(g_\lambda))=H_\lambda\rtimes\ZZ/2\ZZ.\] 
Thus, for $h\in \tilde H_\lambda$ one has $\Ad(h)(g_\lambda)=\mu g_\lambda$ for some $\mu\in\CC\setminus\{0\}$. 
In view of the uniqueness of the Jordan decomposition we have $\Ad(h)(f)=\mu f$ and $\Ad(h)(n)=\mu n$.
It follows that $\tilde H_\lambda\subset\tilde L$ and $H_\lambda\subset L$. 

The centralizer ${\rm Cent}_{\fg}(g_\lambda)$ is contained in $\Lie({\rm Stab}_G(\PP g_\lambda))$. Hence also 
\[\ff\subseteq\Lie({\rm Stab}_G(\PP g_\lambda))={\Lie}(H_\lambda).\] 
Both $\ff$ and ${\Lie}(H_\lambda)$ are Lie subalgebras of dimension 2, hence they coincide. 
Thus, ${\Lie}(H_\lambda)=\ff$ does not depend on $\lambda$, and therefore $H_\lambda=H\subset L$ for all 
$\lambda\in\CC\setminus\{0\}$ where $H:=H_1$. 

The generator $\tau\in G$  of the factor $\ZZ/2\ZZ$ in the decomposition $\tilde H_\lambda=H\rtimes\ZZ/2\ZZ$ stabilizes the point $\PP g_1$, that is, $\Ad(\tau)(f+n)=\pm(f+n)$. So, $\Ad(\tau)(f)=\pm f$ and $\Ad(\tau)(n)=\pm n$ with the same signs in both cases.  
Therefore, $\Ad(\tau)(g_\lambda) = \pm g_\lambda$ for all $\lambda\in\CC$. The latter means that $\tau$ fixes any point 
$\PP g_\lambda\in D_\s$ and acts via involution on the Fano--Mukai fourfold $V(g_\lambda)$ interchanging the pairs of disjoint $\Aut^0(V(g_\lambda))$-invariant cubic cones on $V(g_\lambda)$. Furthermore, $\tau\in\tilde L$ serves as a generator of the factor $\ZZ/2\ZZ$ in the decomposition $\tilde L=L\rtimes\ZZ/2\ZZ$. 
Due to Theorem~\ref{thm:discrete}(iv) and Corollary~\ref{cor:aut-GL2}, up to conjugation by an element of $L$, $\tau$ 
acts on $L=\GL(2,\CC)$ via the Cartan involution $g\mapsto (g^t)^{-1}$ and acts on $H$ via the inversion. Thus, $H$ coincides with one of the subgroups $H^\pm$ from Proposition~\ref{prop:aut-GL2}(v). 

Finally, the involution $\tau$ acts on the total space of the smooth family $\mathscr{F}_\omega\to\omega$ preserving each fiber $V(g_\lambda)$ with $\PP g_\lambda\in \omega$. The fixed point set $(\mathscr{F}_\omega)^\tau$ is a smooth surface fibered into a family of smooth reduced curves of degree 12, each curve being the union of two disjoint rational components $E_1(\lambda)$ and $E_2(\lambda)$. The pair of degrees of these curves is locally constant in the family and equals $(6,6)$ for $\lambda=0$, hence also for any $\lambda\in\CC$. Since the fourfolds $V(h)$ with $\PP h\in\Omega^\aaa$  are pairwise isomorphic and $\omega^\aaa\subset\Omega^\aaa$, the same conclusion holds for any $\PP h\in\Omega^\aaa$. 
\end{proof}

\section{Smoothness of $\Pi$}\label{sec:smooth}
Due to the following proposition $\Pi$ is smooth, so the bijective normalization morphism $\tilde\Pi\to\Pi$ is an isomorphism, see Proposition \ref{prop:bijection} (ii).
\begin{prop}\label{prop:smoothness}
Given a Fano--Mukai fourfold $V$ of genus 10 and an involution $\tau\in\Aut(V)\setminus\Aut^0(V)$, the scroll $\Pi$  in $\tau$-invariant lines on $V$ is smooth. 
\end{prop}
The proof of Proposition \ref{prop:smoothness} is preceded by Lemmas \ref{lem:0-dim} and \ref{lem:proj}. 
\begin{lem}\label{lem:0-dim} $\,$
\begin{enumerate}\item[(i)] $(\sing \Pi)\cap (E_1\cup E_2)=\emptyset$.
\item[(ii)] $\dim(\sing\Pi)=0$.
\end{enumerate}
\end{lem}
\begin{proof}
(i) Pick a point $\tilde P\in \tilde E_i$ and let $\tilde r$ be the ruling of $\tilde\Pi$ passing through $\tilde P$ and $r=s(\tilde r)$ be the ruling of $\Pi$ passing through $P=s(\tilde P)$. The restrictions  $s|_{\tilde E_i}\colon\tilde E_i\to E_i$ and $s|_{\tilde r}\colon\tilde r\to r$ are isomorphisms. Hence the rank of the differential $d(s|_{\tilde \Pi})_{\tilde P}\colon T_{\tilde P}\tilde \Pi\to T_P\Pi$ equals 2, so $s|_{\tilde\Pi}$ is \'etale at $\tilde P$. Now (i) follows. 

(ii) Assume that $\sing\Pi$ contains an irreducible curve $\gamma$, and let $\tilde\gamma=(s|_{\tilde\Pi})^{-1}(\gamma)$. 
By virtue of (i), $\tilde\gamma$ does not intersect $\tilde E_1\cup\tilde E_2$. Then $\deg(\tilde\gamma)=6$ and $\tilde\gamma$ is a constant section of the projection 
$\tilde\Pi\cong\PP^1\times\PP^1\stackrel{{\rm pr}_1}{\longrightarrow}\PP^1$. The morphism $s|_{\tilde\gamma}\colon\tilde\gamma\to\gamma$ is bijective, so \'etale at a general point $\tilde Q\in\tilde\gamma$. By the same argument as in the proof of (i), $s|_{\tilde\Pi}\colon\tilde\Pi\to\Pi$ is \'etale at $\tilde Q$. This contradicts our assumption $\gamma\subset\sing\Pi$. 
\end{proof}
In the sequel we use the following notation.
\begin{nota}  
The complete linear system of type $(1,6)$ on $\PP^1\times\PP^1$ embeds the quadric $\PP^1\times\PP^1$ 
onto a rational normal scroll $\Pi_{\mathrm{norm}}\subset \PP^{13}$ of degree 12. The fibers of the projection ${\rm pr}_1\colon\PP^1\times\PP^1\to\PP^1$ are sent to rulings of $\Pi_{\mathrm{norm}}$ and
the constant sections of  ${\rm pr}_1$ are sent to smooth rational 
normal sextic curves $\bar \gamma_t$ on $\Pi_{\mathrm{norm}}$. Any two different curves $\bar \gamma_t$ and $\bar \gamma_{t'}$ span two disjoint $\PP^6$'s in $\PP^{13}$. We identify $\Pi_{\mathrm{norm}}$ with the normalization $\tilde\Pi$ of $\Pi$. Thus, the linear projection  $\pi\colon \PP^{13}\to\PP^k:=\langle\Pi\rangle\subset\PP^{12}$ with center $\PP^{12-k}\subset \PP^{13}$ defined by the
linear system of hyperplane sections of $\Pi\subset \PP^{12}$ is the normalization morphism. 
By Proposition \ref{prop:bijection}(ii) $\pi$ is a bijection. So, $\PP^{12-k}$ intersects no non-tangent secant line to $\Pi_{\mathrm{norm}}$. Furthermore, it does not meet $\Pi_{\mathrm{norm}}$ because $\deg(\Pi_{\mathrm{norm}})=\deg(\Pi)=12$.

Following the notation in \cite{Har92}, for a point $\bar Q\in \Pi_{\mathrm{norm}}$ we let $T_{\bar Q}\Pi_{\mathrm{norm}}$ and
$\TT_{\bar Q}\Pi_{\mathrm{norm}}\subset\PP^{13}$ be the tangent plane and the projective tangent plane of $\Pi_{\mathrm{norm}}$ at $\bar Q$, respectively. Clearly, the image $Q=\pi(\bar Q)$ is a singular point of $\Pi$ if and only if the center $\PP^{12-k}$ of $\pi$ intersects the projective tangent plane $\TT_{\bar Q}(\Pi_{\mathrm{norm}})$. Recall the following   well known facts. 
\end{nota}  
\begin{lem}\label{lem:proj} $\,$
\begin{enumerate}
\item[{\rm (i)}] $\Pi_{\mathrm{norm}}$ is cut out by quadric hypersurfaces in $\PP^{13}$.
\item[{\rm (ii)}] Every projective line in $\PP^{13}$ tangent  to $\Pi_{\mathrm{norm}}$ and different from its rulings intersects $\Pi_{\mathrm{norm}}$ in a single point.
\item[{\rm (iii)}]  Every projective plane in $\PP^{13}$ tangent  to $\Pi_{\mathrm{norm}}$ intersects $\Pi_{\mathrm{norm}}$ in a single ruling. 
\item[{\rm (iv)}] Given a ruling $\bar r$ of $\Pi_{\mathrm{norm}}$, the projective planes tangent  to $\Pi_{\mathrm{norm}}$ along $\bar r$ sweep out a linear subspace $\PP^3$ of $\PP^{13}$ that contains $\bar r$. 
\item[{\rm (v)}]  Projective planes tangent to $\Pi_{\mathrm{norm}}$ sweep out a fourfold ${\rm Tan}(\Pi_{\mathrm{norm}})$ (also swept out by the one-parameter family of $ \PP ^3$'s from (iv)).
\item[{\rm (vi)}] The secant variety ${\rm Sec}(\Pi_{\mathrm{norm}})$ is a fivefold of degree $45$.
\end{enumerate}
\end{lem}
\begin{proof} Statement (i) follows from \cite[9.11--9.12]{Har92} and implies (ii), which in turn implies (iii). 
See \cite[Section I.51]{Edge31} for (iv)-(vi). 
\end{proof}
\begin{proof}[\textit{Proof of Proposition \ref{prop:smoothness}}] By the universal property of normalization, the action of $\tau$ on $\Pi$ lifts to a $\tau$-action on the normalization $\Pi_{\mathrm{norm}}$, which in turn extends to the ambient space $\PP^{13}$ as a linear involution. Let $\bar E_i$ be the preimage of $E_i$ in $\Pi_{\mathrm{norm}}$ and let $\bar r$ be the preimage of a ruling $r$ of $\Pi$. The fixed point set of $\tau$ acting on $\PP^{13}$ is the union of two disjoint $\PP^6$s that are the linear spans $\langle \bar E_1\rangle$  and $\langle \bar E_2\rangle$. Any ruling $\bar r$ of $\Pi_{\mathrm{norm}}$ is $\tau$-invariant. 
The projection $\pi\colon\PP^{13}\dasharrow\PP^k$ is $\tau$-equivariant. 

If  $Q\in \Pi$ is a singular point, then $Q\notin E_1\cup E_2$, see Lemma \ref{lem:0-dim}(i). So, $\tau(Q)\neq Q$ is a second singular point of $\Pi$. The ruling $r$ of $\Pi$ that contains $Q$ is $\tau$-invariant, therefore $r=(Q\tau(Q))$. 
Let $\bar Q\in \Pi_{\mathrm{norm}}$ be the preimage of $Q$, $\bar r$ be the ruling of $\Pi_{\mathrm{norm}}$ passing through $\bar Q$ and $\tau(\bar Q)=\overline{\tau(Q)}\in\bar r$ be the preimage of $\tau(Q)$. The projective tangent planes $\TT^2_{\bar Q}\Pi_{\mathrm{norm}}$ and $\TT^2_{\tau(\bar Q)}\Pi_{\mathrm{norm}}$ intersect precisely along $\bar r$ and span a $\PP^3\subset\PP^{13}$ that contains each projective tangent plane $\TT_{\bar Q'}\Pi_{\mathrm{norm}}$ for $\bar Q'\in\bar r$, see Lemma \ref{lem:proj}(iv).  

The center $\PP^{12-k}$ of $\pi$ is $\tau$-invariant. Since $Q\in\Pi$ is a singular point,  $\PP^{12-k}$ meets $\TT^2_{\bar Q}\Pi_{\mathrm{norm}}$ at a point, say $\bar P\neq \bar Q$. Since $\tau(\bar Q)$ is singular on $\Pi$, $\PP^{12-k}$ also meets $\TT^2_{\tau(\bar Q)}\Pi_{\mathrm{norm}}$ at $\tau(\bar P)$.  The projective line
\begin{equation}\label{eq:line} l_Q=(\bar P, \tau(\bar P))=\PP^{12-k}\cap \PP^3\end{equation}
does not meet $\bar r$ and intersects each projective tangent plane $\TT_{\bar Q'}\Pi_{\mathrm{norm}}$ from the pencil of planes in $\PP^3$ with base locus $\bar r$. It follows that each point $Q'=\pi(\bar Q')\in r$ is a singular point of $\Pi$. The latter contradicts Lemma \ref{lem:0-dim}(ii).
\end{proof}
\section{Linear non-degeneracy of $\Pi$ in $\PP^{12}$}
What is left to complete the proof of Theorem \ref{thm:mthm}? We do not yet know the value of $k=\dim\langle \Pi\rangle$
and the splitting types of the normal bundles of the fixed curves $E_1$ and $E_2$.
The first is done in this section and the second in the next.

\subsection{First observations}\label{ss:1st}
 Let as before $\fg$ be the Lie algebra of $G$, $T$ be a fixed maximal torus of $G$, $\fh=\Lie(T)$ be the corresponding Cartan subalgebra of $\fg$, and $\Delta$ be the root system in $\fh^\vee$.
For a root $\boldsymbol\alpha\in \Delta$ 
let $\fg_{\boldsymbol\alpha}$ be the corresponding root subspace of $\fg$. 
Consider the Cartan decomposition 
\begin{equation*}
\fg=\fh
\oplus \bigoplus_{\boldsymbol\alpha\in\Delta^+}(\mathfrak{g}_{
\boldsymbol\alpha}
\oplus\mathfrak{g}_{-\boldsymbol\alpha}),
\end{equation*}
and let  $W=N_{G}(T)/T$  be the Weyl group of $G$. The central Weyl involution $\tau\in z(\W)\subset G$ acts on $\fh$ via $\tau|_{\fh}=-{\rm id}_{\fh}$ and acts on every plane 
$\mathfrak{g}_{\boldsymbol\alpha}\oplus\mathfrak{g}_{-\boldsymbol\alpha}$ 
by interchanging $\mathfrak{g}_{\boldsymbol\alpha}$ and $\mathfrak{g}_{-\boldsymbol\alpha}$. So $\fg$ can be decomposed into a direct sum $\fg=F^+\oplus F^-$ of eigenspaces orthogonal with respect to the Killing form, where
\begin{equation}\label{eq:dim} \tau|_{F^+}={\rm id}_{F^+},\quad \tau|_{F^-}=-{\rm id}_{F^-},\quad\dim F^+=6, \quad\dim F^-=8,\quad \fh\subset F^-.\end{equation}
 The fixed point set $(\PP\fg)^\tau\subset\PP^{13}$ consists of two disjoint subspaces 
 $\PP F^+\cong\PP^5$ and $\PP F^-\cong\PP^7$, where $\PP\fh\subset\PP F^-$. 
 
Since the torus $T$ is an abelian group, $T|_{\fh}=\id_{\fh}$ and $T|_{\PP\fh}=\id_{\PP\fh}$. The center $z(\W)=\langle\tau\rangle$ of the Weyl group also acts identically on $\PP\fh$. So, the action of the dihedral group $\W=\fD_6$ on $\fh$ descends to an effective action of the symmetric group $\Sym_3=\W/z(\W)$ on the projective line $\PP\fh\subset\PP\fg=\PP^{13}$. This line meets the sextic hypersurfaces $D_{\ell}$ and $D_{\mathrm{s}}$ along the zero-cycles $2\Orb_{\ell}$ and $2\Orb_{\mathrm{s}}$, respectively, where $\Orb_{\ell}$ and $\Orb_{\mathrm{s}}$ are $\Sym_3$-orbits of length 3, see \cite[Lemmas 3.7 and 3.8]{PZ22}. 

 Choose a point $\PP h\in\PP\fh\setminus D_\ell$. Then the hyperplane section $V=V(h)=h^\bot\cap\Omega$ is a smooth Fano--Mukai fourfold of genus 10. The group $\Aut(V)$ is reductive  and contains the torus $T$ and the involution $\tau$ normalizing $T$. Notice that $V=V(h)$ is singular exactly when $\PP h\in \Orb_{\ell}=\PP h\cap D_\ell$.
 
Let now $V$ be a smooth Fano--Mukai fourfold of genus 10 with a reductive group $\Aut(V)$ that contains the torus $T$ and the involution $\tau$ normalizing $T$. Then $V\cong V(h)$ for some $h\in\fg$ with $\PP h\in\PP\fh\setminus D_\ell$, see \cite[Corollary 3.5.1]{PZ22}. Besides, $\Aut^0(V)\cong \GL_2(\CC)$ if $\PP h\in\Orb_{\mathrm{s}}$ and 
 $\Aut^0(V)=T\cong\Gm^2$ if $\PP h\in\PP\fh\setminus (\Orb_{\mathrm{s}}\cup  \Orb_{\ell})$, 
 see  \cite[Corollary 3.8.1]{PZ22}. 
 
It is known that $V=V(h)\subset \PP^{12}$ is linearly nondegenerate. Since the involution $\tau$ leaves $\langle V(h)\rangle=h^\bot$ invariant, we have $h\in F^+\cup F^-$. Letting  
\[\Lambda^+=\PP F^+\cap h^\bot\quad\text{and}\quad \Lambda^-=\PP F^-\cap h^\bot\]
we have 
\[\Lambda^+\cap\Lambda^-=\emptyset,\quad\dim\Lambda^++\dim\Lambda^-=11\quad \quad\text{and}\quad\langle \Lambda^+\cup\Lambda^-\rangle=h^\bot\cong\PP^{12}.\]
For $h\in\fh\subset F^-$ one has $h^\bot\supset F^+$, therefore
$ \Lambda^+\cong\PP^{5}$ and $\Lambda^-\cong\PP^{6}$.

Let $\Pi\subset V=V(h)$ be the scroll in $\tau$-invariant lines. 
Then $\Pi^\tau=E^+\cup E^-$ where $E^+$ and 
$E^-$ are disjoint smooth rational sextic curves, see Theorem \ref{thm:towards-mthm}. We may consider that $E^+\subset \Lambda^+$ and 
$E^-\subset \Lambda^-$.  Indeed, if both $E^+$ and $E^-$ were contained in the same projectivized eigenspace, say 
$E^+\cup E^-\subset\Lambda^-$, then 
$\Pi\subset\langle E^+ \cup E^-\rangle\subset \Lambda^-$, and therefore
$\tau|_\Pi={\rm id}_{\Pi}$, a contradiction.  Hence 
$E^\pm=\Lambda^\pm\cap V$. Since 
$\dim \Lambda^+=5$, the sextic $E^+\subset\Lambda^+$ is not a rational normal curve and
$\Pi\subset\PP^{12}$ is not linearly normal. 
\subsection{Linear non-degeneracy of $\Pi$}
\begin{prop}\label{prop:lin-nondegen} $\,$
\begin{enumerate}
\item[{\rm (i)}] The scroll $\Pi$ in $\tau$-invariant lines is linearly nondegenerate in $\PP^{12}$.
\item[{\rm (ii)}] Up to enumeration we have $\dim \langle E_1\rangle = 5$ and $\dim \langle E_2\rangle = 6$. 
\end{enumerate}
\end{prop}
We start with the following lemma.
\begin{lem}\label{lem:join} Let $k=\dim\langle\Pi\rangle$ and let $\pi\colon\PP^{13}\dasharrow\PP^k=\langle\Pi\rangle$ be the linear projection  with center  
$Z=\PP^{12-k}$ which sends the rational normal scroll $\Pi_{\mathrm{norm}}$ onto $\Pi$ and sends $\bar E_i$ onto $E_i$, $i=1,2$. Let $H_i=\langle\bar E_i\rangle\subset\PP^{13}$ and $Z_i=Z\cap H_i$, $i=1,2$. Then either
\begin{itemize}\item $Z=Z_i$ for some $i\in\{1,2\}$, or
\item $Z$ coincides with the join $J(Z_1, Z_2)=\langle Z_1,Z_2\rangle$.
\end{itemize}
In the latter case $\dim(Z)=12-k=\dim(Z_1)+\dim(Z_2)+1$. 
\end{lem}
\begin{proof} 
We have $H_i\cong\PP^6$, $H_1\cap H_2=\emptyset$ and $H_1\cup H_2$ is the fixed point set $(\PP^{13})^{\tau}$.  Recall that through any point of $\PP^{13}\setminus (H_1\cup H_2)$ passes a unique line meeting both $H_1$ and $H_2$. 

Assume that $Z\neq Z_i$ for $i=1,2$ and let $P\in Z\setminus (H_1\cup H_2)$.  Since $Z$ is $\tau$-invariant, the $\tau$-invariant line $l_P=(P\tau(P))$ is contained in $Z$ and $P_i=l_P\cap H_i\in Z_i$, $i=1,2$, are the two fixed points of $\tau$ on $l_P$.
Thus, $P\in l_P=(P_1P_2)\subset J(Z_1,Z_2)$, and so $Z\subset  J(Z_1,Z_2)$.   The opposite inclusion is evidently true.
\end{proof}
\begin{cor}\label{cor:lin-nondegen}
The scroll $\Pi$ is linearly nondegenerate if and only if $Z$ is a point in $H_i\setminus {\rm Sec}(\bar E_i)$ for some $i\in\{1,2\}$. 
In the latter case 
\[\dim\,\langle E_i\rangle=5\quad\text{and}\quad\dim\,\langle E_j\rangle=6\quad\text{for}\quad j\neq i.\]
\end{cor}
\begin{proof}  
This follows immediately from Lemma \ref{lem:join}. 
\end{proof}
 Let us recall the notation and some formulas from \cite{Hir57}.
\begin{nota}\label{nota:Hironaka} 
Let $C_1$ and $C_2$ be curves in $\PP^n$ with no common component. For a point $P\in C_1\cap C_2$ let $I(C_i)$ be the ideal of $C_i$ in the local ring $A$ of $(\PP^n,P)$. Then the intersection index of $C_1$ and $C_2$ at $P$ is defined as
\[i(C_1,C_2;P)={\rm length}\,A/(I(C_1),I(C_2))A.\]
It is known that $i(C_1,C_2;P)=1$ if and only if the Zariski tangent spaces of $C_1$ and $C_2$ at $P$ intersect only in $P$, see \cite[Proposition 3]{Hir57}. 

Let now $C_1=C_1'\cup C_1''$ where $C_1'$, $C_1''$ and  $C_2$ are smooth at $P$. If $C_1$ lies on  a smooth surface $S$ and $C_2$ is transversal to $S$ at $P$ then $i(C_1,C_2;P)=1$. If $C_1'$ is a simple tangent line to $C_2$ at $P$ and $C_1''$ is transversal to $C_2$ at $P$ then $i(C_1,C_2;P)=3$. If, finally, $C_1'$ is a simple tangent line to $C_2$ at $P$ and $C_1''=C_1'$ then $i(C_1,C_2;P)=2i(C_1',C_2;P)=4$. 

Given $r$ curves $C_1,\ldots,C_r$ with no common component, one considers the following zero-cycle supported at the singular points of $\cC=\bigcup_{i=1}^r C_i$:
\[\Lambda (\cC)=\sum_{P\in {\rm sing}(\cC)} i(\Lambda(\cC), P)P\quad\text{where}\quad  i(\Lambda(\cC), P)=
\sum_{k=2}^r i\left(\left(\sum_{i=1}^{k-1} C_i\right), C_k; P\right).\]
In fact, this cycle does not depend on the choice of enumeration.
The arithmetic genus $p_a(\cC)$ can be computed by the following formula, 
see \cite[Theorem 3]{Hir57}:
\begin{equation}\label{eq:Hironaka}
p_a(\cC)=\sum_{i=1}^r p_a(C_i)+\deg (\Lambda(\cC))-(r-1).
\end{equation}
\end{nota}
\begin{lem}\label{lem:11} $\,$
\begin{enumerate}
\item[{\rm (i)}]
 For $i=1,2$ the curve $E_i$ is intersection of quadrics in $\langle E_i\rangle$. 
\item[{\rm (ii)}] We have $k\ge 11$. 
\end{enumerate}
\end{lem}
\begin{proof}  (i) It is known that $\Omega\subset \PP^{13}$ and $V=\Omega\cap h^\bot$ in $\PP^{12}$ are intersections of quadrics, see \cite[Lemma 2.10]{Isk77}, cf. also \cite[4.5]{PZ18}. Then also 
$E_i=E^\pm=\Lambda^\pm\cap V=\langle E_i\rangle\cap V$ is intersection of quadrics in $\langle E_i\rangle$, $i=1,2$, see Subsection \ref{ss:1st}.

(ii) Fix $i\in\{1,2\}$ and let $E=E_i$, $\bar E=\bar E_i$ and $\mathcal{Z}=Z_i\subset\langle \bar E\rangle=\PP^6$. 
Since the curve $E=\pi(\bar E)$ is smooth, the center $\mathcal{Z}$ of the projection $\pi: \PP^6\dasharrow\langle E\rangle\subset\PP^{12}$ does not intersect the threefold ${\rm Sec}(\bar E)$ in $\PP^6$. Hence $\dim \mathcal{Z}\le 2$. 

If $\dim \mathcal{Z}=2$
then $\langle E\rangle=\PP^3$, and therefore $E$ is intersection of quadrics in $\PP^3$, see (i). Thus,  $E\subset Q_1\cap Q_2$ for some distinct quadrics $Q_1$ and $Q_2$  in $\PP^3$. Hence
$6=\deg(E)\le 4$, a contradiction. 
It follows that $\dim  \mathcal{Z} \le 1$. 

If $\dim  \mathcal{Z}=1$ then $\langle E\rangle=\pi(\langle\bar E\rangle)=\PP^4$. By (i), $E$ is contained in a proper intersection of  quadrics $Q_1, Q_2$ and $Q_3$ in $\PP^4$. This intersection is a curve $\Gamma$ of degree $8$ and of arithmetic genus $5$. It consists of the smooth sextic $E$ and a conic $C_0$. We claim that $C_0$ is reducible. 
Indeed, let $L=\PP^3$ be a hyperplane in $\PP^4$ that contains $C_0$. Then the surface $L\cap {\rm Sec}(E)$ in $L=\PP^3$ meets $C_0$. 
Hence, there is a secant line $l$ of $E$ which intersects $C_0$, and therefore $l\cdot Q_i\ge 3$ for $i=1,2,3$. 
Thus, $l\subset\bigcap_{i=1}^3 Q_i$
 is a component of $C_0=l+l'$ where $l'$ is a line. 
 
For $\Gamma=E+l+l'$  we obtain by \eqref{eq:Hironaka}
 \[5=\pi_a(\Gamma)=\deg(\Lambda(\Gamma))-2.\]
So, $\deg(\Lambda(\Gamma))=7$. 
However, we claim that $\deg(\Lambda(\Gamma))<7$ whatever is the configuration $(E, l, l')$. 

Indeed, $E$ being a smooth intersection of quadrics, the intersection of $E$ 
 with any line is either transversal or a simple tangency at a single point. Therefore, $\sum_{P\in E} i(E,l_i;P)\le 2$. 
 Assume first that $l_1\neq l_2$. If $\Gamma$ has only nodes as singularities then the number of these points is $\le 5$, 
 and also $\deg(\Lambda(\Gamma))\le 5$, a contradiction. 
 
 Suppose now that $\Gamma$ has a triple point $P$. Then either $l_1, l_2$ and $E$ are transversal at $P$ and then 
 \[i(\Lambda(\Gamma), P)=i(l_1,l_2;P)+i(l_1+l_2, E;P)=2,\]
 or at most one of the $l_i$, say $l_1$ is a simple tangent to $E$ at $P$ and $l_2$ is transversal to $l_1+E$ at $P$, hence
 \[i(\Lambda(\Gamma), P)=i(E,l_1;P)+i(E+l_1,l_2;P)=3.\]
 It is easily seen that in any case $\deg(\Lambda(\Gamma))\le 5$. This leads again to a contradiction. 
 
Suppose finally that $l_1=l_2$ is a double line of the intersection $Q_1\cdot Q_2\cdot Q_3$. If this double line meets $E$ transversally in two points $P_1$ and $P_2$, then  $\deg(\Lambda(\Gamma))=2(i(l_1,E;P_1)+i(l_1,E;P_2)=4$. If $l_1$ and $E$ are tangent at a point $P$ then 
 again $\deg(\Lambda(\Gamma))=2i(l_1,E;P)=4$. In all the other cases $\deg(\Lambda(\Gamma))<4$, which proves our claim and excludes the case $\dim  \mathcal{Z}=1$. 

Thus, $\dim Z_i\le 0$ for $i=1,2$, hence  $12-k=\dim Z\le 1$ and $k\ge 11$.
\end{proof}
\begin{proof}[Proof of Proposition \ref{prop:lin-nondegen}]
Statement (ii) follows from (i) due to Corollary \ref{cor:lin-nondegen}. To show (i) suppose that $\Pi$ is linearly degenerate. By Lemma \ref{lem:11} one has $\langle\Pi\rangle=\PP^{11}$, and therefore $\Pi$ is contained in the hyperplane section $M=\langle\Pi\rangle\cap V$.
By the Lefschetz hyperplane theorem, $\Pic(M)\cong\Pic(V)=\ZZ$, see e.g. \cite[Section 1]{Fuj80}. Hence $\Pi$ is a complete intersection in $M$. However, $\deg(\Pi)=12$ does not divide $\deg(M)=18$, a contradiction. 
\end{proof}
\section{Normal bundles of the fixed curves}
In the notation of Subsection \ref{ss:1st} for  the components $E^\pm$ of the fixed point set $V^\tau$ we have $\langle E^+\rangle =5$ and $\langle E^-\rangle =6$, see Proposition \ref{prop:lin-nondegen}.
In this section we determine the splitting type of the normal bundles $N^-=N_{E^-/\PP^6}$ and $N^+=N_{E^+/\PP^5}$. 
We show that the first normal bundle is almost balanced and the second is balanced. More precisely, we have the following
\begin{prop}\label{prop:norm-bdl} The normal bundles $N^\pm$ admit decompositions 
\[N^-=\cO_{\PP^1}(8)^{\oplus 5}\quad\text{and}\quad N^+=\cO_{\PP^1}(8)^{\oplus 2}\oplus\cO_{\PP^1}(9)^{\oplus 2}.\]
\end{prop}
Actually, the first equality is well known. Indeed, let $\cC\subset \PP^n$ be a linearly nondegenerate smooth rational curve of degree $d\ge n$.
Then 
\[N_{\cC/\PP^n}=\bigoplus_{i=1}^{n-1}\cO_{\PP^1}(d+b_i)\quad\text{where}\quad b_i\ge 2\quad\text{and}\quad \sum_{i=1}^{n-1} b_i=2d-2,\]
see e.g. \cite{Sac80}, \cite[Corollary 2.2]{CR18}. In particular, for the rational normal curve $\cC=C_n$ of degree $n$ in $\PP^n$ (i.e. for $d=n$) one has $b_1=\ldots=b_{n-1}=2$. So, the normal bundle is balanced, that is,
\[N_{C_n/\PP^n}=\cO_{\PP^1}(n+2)^{\oplus\, n-1}.\]
For $n=6$ this yields the first equality in Proposition \ref{prop:norm-bdl}. 

As for the second, notice that 
the preceding formulas leaves just the following possibilities (a) and (b), depending on the position of the center $\cZ$ of the linear projection $\pi\colon\PP^6\dasharrow\PP^5$ which sends a rational normal curve $C_6$ to $E^+$:
\begin{equation}\label{eq:split} \text{(a)}\quad N^+=\cO_{\PP^1}(8)^{\oplus 2}\oplus\cO_{\PP^1}(9)^{\oplus 2}\quad\text{and}\quad\text{(b)}\quad N^+=\cO_{\PP^1}(8)^{\oplus 3}\oplus\cO_{\PP^1}(10).\end{equation}
A priori, both of these splitting types could occur, due to the following facts. 
\begin{thm}\label{thm:Sacchiero} $\,$
\begin{enumerate}
\item[{\rm (i)}] {\rm (\cite{Sac80}; see also \cite[Theorem 2.7]{CR18})}
For any sequence of integers $b_i \ge 2$, $1 \le i \le n-1$ such that $\sum_{i=1}^{n-1} b_i = 2d-2$ 
there exists an
unramified map $f \colon \PP^1 \to \PP^n$ onto a linearly nondegenerate, immersed curve $\cC\subset\PP^n$ 
of degree $d\ge n$ 
such that \[N_f=\bigoplus_{i=1}^{n-1} \cO_{\PP^1}(d+b_i).\] 
\item[{\rm (ii)}] {\rm (\cite{Sac80}; see also \cite[Theorem 2]{AR17a})} 
For a generic rational curve $\cC$ in $\PP^n$ of degree $d>n\ge 3$ one has
\[N_f=\cO_{\PP^1}(d+q+1)^{\oplus n-r-1}\oplus \cO_{\PP^1}(d+q+2)^{\oplus r}\] 
where $2d-n-1=q(n-1)+r$ with $r<n-1$. 
\end{enumerate}
\end{thm}
From (ii) of Theorem \ref{thm:Sacchiero} we deduce the following
\begin{cor}\label{cor:Sacchiero} The normal bundle of a generic rational sextic curve in $\PP^5$ is of splitting type (a) in \eqref{eq:split}. 
\end{cor}
Let us give concrete examples of rational sextic curves with normal bundles of splitting types (a) and (b)  in \eqref{eq:split}. 
\begin{exa}\label{ex:monomial}
Consider the smooth monomial sextic curves $\cC$ and $\cC'$ in $\PP^5$ with parametrization
\[\cC=(u^6:u^5v:u^4v^2:u^2v^4:uv^5:v^6) \quad\text{resp.}\quad \cC'=(u^6:u^5v:u^3v^3:u^2v^4:uv^5:v^6). \]
Then 
\[N_{\cC/\PP^5}\cong N_{\cC'/\PP^5}\cong \cO_{\PP^1}(8)^{\oplus 2}\oplus\cO_{\PP^1}(9)^{\oplus 2},\] 
see \cite[Theorem 3.2]{CR18}. 

The splitting type (b) in \eqref{eq:split} can be realized for immersed rational sextic curves. 
For example, let $p_i=a_iu+b_iv$, $i=1,2$ be  two coprime linear forms with $a_i, b_i\neq 0$, 
and let $\cC\subset\PP^5$ be the rational sextic curve with parameterization
\[f\colon (u:v)\mapsto (u^5p_1:u^4vp_1:u^3v^2p_1:u^2v^3p_1:uv^4p_1:v^5p_2).\] 
Then $\cC$ is immersed with normal bundle
\[N_f=\cO_{\PP^1}(8)^{\oplus 3}\oplus\cO_{\PP^1}(10),\] 
see \cite{Sac80} and \cite[Lemma 2.4]{CR18}. Notice that $\cC$ is the image of the rational normal curve 
$C_6=h(\PP^1)\subset\PP^6$ parameterized via
\[(u:v)\mapsto (u^6:u^5v:\ldots:v^6)\] 
under the projection 
$\pi_P\colon\PP^6\dasharrow \PP^5$ with center 
\[P=(1:\lambda:\ldots:\lambda^5:\lambda^5\mu)\in \PP^6
\quad\text{where}\quad \lambda=-a_1/b_1\quad\text{and}\quad \mu=-a_2/b_2.\]
The center $P$ is situated on the secant line $(Q_1(\lambda)Q_2)$ of $C_6$ where 
\[Q_1(\lambda)=(1:\lambda:\ldots:\lambda^6)\in C_6\quad\text{and}\quad Q_2=(0:0:\ldots:0:1)\in C_6.\]
Hence $\pi(Q_2)=(0:\ldots:0:1)\in\PP^5$ is a point of self-intersection of $\cC=\pi_P(C_6)$. 
\end{exa}
The following theorem allows to replace ``generic''  by ``smooth'' in Theorem \ref{thm:Sacchiero}(ii) 
and Corollary \ref{cor:Sacchiero}. 

\begin{thm}[{\rm \cite[Theorem 3.3.13]{Ber11}, \cite[Theorem 2.12]{Ber12}\footnote{Cf.\ also 
\cite[Theorem 4.5]{Ber14}.}}]\label{thm:Bernardi}
Let $C_n \subset \PP^n$ be the rational normal curve of degree $n$. 
Then the image $\cC=\pi_P(C_n) \subset \PP^{n-1}$ has normal bundle 
\[N_{\cC/\PP^{n-1}} = \cO_{\PP^1}(n+2)^{\oplus\, n-4} \oplus \cO_{\PP^1}(n+3)^{\oplus 2}\] 
if and only if the center $P\in\PP^n$ of the projection $\pi_P\colon\PP^n\dasharrow\PP^{n-1}$ 
does not lie on a secant or tangent line to $C_n$, if and only if $\cC$ is smooth.
\end{thm}
Since $E^+\subset\PP^5=\langle E^+\rangle$ is a smooth sextic curve, the following corollary is immediate.
\begin{cor}\label{cor:second-case} One has
\[N_{E^+/\PP^{5}} = \cO_{\PP^1}(8)^{\oplus 2} \oplus \cO_{\PP^1}(9)^{\oplus 2}.\] 
\end{cor}
This ends the proofs of Proposition \ref{prop:norm-bdl} and Theorem \ref{thm:mthm}.
%
%
\newcommand{\etalchar}[1]{$^{#1}$}
\def\cprime{$'$}

\end{document}